\documentclass[12pt,a4paper,reqno]{amsart}
\usepackage{amsfonts,amsthm,amsmath,amssymb,%amscd,
	amsbsy,euscript}%,textcomp,url} %,stmaryrd} cannot find "stmaryrd"? %,float} %,textcomp
\usepackage{tikz}
\usepackage{tkz-graph}

\newtheorem{theor}{Theorem}%[section]

\theoremstyle{definition}

\newtheorem{proposition}[theor]{Proposition}%[section]
\newtheorem{lemma}[theor]{Lemma}%[section]
\newtheorem{cor}[theor]{Corollary}%[section]
\newtheorem{define}{Definition}%[section]

\newtheorem*{notation}{Notation}
\newtheorem{example}{Example}%[section]
\theoremstyle{remark}
\newtheorem{rem}{Remark}%[section]

\newcommand{\cEv}{\partial}
\newcommand{\pinner}{\mathbin{\mathchoice
{\hbox{\vrule width0.6em depth0pt height0.4pt
	\vrule width0.4pt depth0pt height0.8ex}}
{\hbox{\vrule width0.6em depth0pt height0.4pt
	\vrule width0.4pt depth0pt height0.8ex}}
{\hbox{\kern0.14em
	\vrule width0.48em depth0pt height0.4pt
	\vrule width0.4pt depth0pt height0.6ex\kern0.14em}}
{\hbox{\kern0.1em
	\vrule width0.39em depth0pt height0.4pt
	\vrule width0.4pt depth0pt height0.5ex\kern0.1em}}}}
\newcommand{\inner}{\pinner\,}

\let \wt=\widetilde
\newcommand{\BBR}{\mathbb{R}}\newcommand{\BBC}{\mathbb{C}}
\newcommand{\BBF}{\mathbb{F}}\newcommand{\BBN}{\mathbb{N}}
\newcommand{\BBS}{\mathbb{S}}\newcommand{\BBT}{\mathbb{T}}
\newcommand{\BBZ}{\mathbb{Z}}\newcommand{\BBE}{\mathbb{E}}
\newcommand{\EuA}{{{\EuScript A}}}
\newcommand{\EuX}{{{\EuScript X}}}
\newcommand{\cA}{{{\EuScript A}}}%{\mathcal{A}}
\newcommand{\bcA}{\boldsymbol{\mathcal{A}}}
\newcommand{\mcA}{\mathcal{A}}
\newcommand{\bcP}{{\boldsymbol{\mathcal{P}}}}
\newcommand{\bcQ}{{\boldsymbol{\mathcal{Q}}}}
\newcommand{\cB}{\mathcal{B}}
\newcommand{\cC}{\mathcal{C}}\newcommand{\tcC}{\smash{\widetilde{\mathcal{C}}}}
\newcommand{\cD}{\mathcal{D}}
\newcommand{\tcX}{\smash{\widetilde{\mathcal{X}}}}
\newcommand{\cE}{\mathcal{E}}\newcommand{\tcE}{\smash{\widetilde{\mathcal{E}}}}
\newcommand{\cEL}{\mathcal{E}_{\IL}}
\newcommand{\cEEL}{{\cE}_{\text{\textup{EL}}}}
\newcommand{\cEKdV}{{\cE}_{\text{\textup{KdV}}}}
\newcommand{\cELiou}{{\cE}_{\text{\textup{Liou}}}}
\newcommand{\cEToda}{{\cE}_{\text{\textup{Toda}}}}
\newcommand{\cF}{\mathcal{F}}
\newcommand{\cH}{\mathcal{H}}
\newcommand{\cN}{\mathcal{N}}
\newcommand{\cI}{\mathcal{I}}
\newcommand{\cJ}{\mathcal{J}}
\newcommand{\cL}{\mathcal{L}}
\newcommand{\cO}{\mathcal{O}}
\newcommand{\cP}{\mathcal{P}}\newcommand{\cR}{\mathcal{R}}
\newcommand{\cQ}{\mathcal{Q}}
\newcommand{\cU}{\mathcal{U}}
\newcommand{\cV}{\mathcal{V}}
\newcommand{\cW}{\mathcal{W}}
\newcommand{\cX}{{\EuScript X}}    %{\mathcal{X}}
\newcommand{\cY}{{\EuScript Y}}    %{\mathcal{Y}}
\newcommand{\cZ}{{\EuScript Z}}    %{\mathcal{Y}}
\newcommand{\boldb}{{\boldsymbol{b}}}
\newcommand{\Bone}{{\boldsymbol{1}}}
\newcommand{\bc}{{\mathbf{c}}}
\newcommand{\ba}{{\boldsymbol{a}}}
\newcommand{\bb}{{\boldsymbol{b}}}
\newcommand{\bbD}{{\boldsymbol{\mathrm{D}}}}
\newcommand{\bbf}{{\boldsymbol{f}}}
\newcommand{\bi}{{\boldsymbol{i}}}
\newcommand{\bn}{{\boldsymbol{n}}}
\newcommand{\bp}{{\boldsymbol{p}}}
\newcommand{\bq}{{\boldsymbol{q}}}
\newcommand{\br}{{\boldsymbol{r}}}
\newcommand{\bs}{{\boldsymbol{s}}}
\newcommand{\bu}{{\boldsymbol{u}}}
\newcommand{\bv}{{\boldsymbol{v}}}
\newcommand{\bw}{{\boldsymbol{w}}}
\newcommand{\bx}{{\boldsymbol{x}}}
\newcommand{\bby}{{\boldsymbol{y}}}
\newcommand{\bz}{{\boldsymbol{z}}}
\newcommand{\bA}{{\boldsymbol{A}}}
\newcommand{\bD}{{\boldsymbol{D}}}
\newcommand{\bF}{{\boldsymbol{F}}}
\newcommand{\bH}{{\boldsymbol{H}}}
\newcommand{\bL}{{\boldsymbol{L}}}
\newcommand{\bN}{{\boldsymbol{N}}}
\newcommand{\bP}{{\boldsymbol{P}}}
\newcommand{\bQ}{{\boldsymbol{Q}}}
\newcommand{\bbU}{{\boldsymbol{U}}}
\newcommand{\bV}{{\boldsymbol{V}}}
\newcommand{\bE}{\mathbf{E}}
\newcommand{\bR}{\mathbf{r}}
\newcommand{\bS}{{\boldsymbol{S}}}
\newcommand{\bal}{{\boldsymbol{\alpha}}}
\newcommand{\bpi}{{\boldsymbol{\pi}}}
\newcommand{\bpsi}{{\boldsymbol{\psi}}}
\newcommand{\bxi}{{\boldsymbol{\xi}}}
\newcommand{\bet}{{\boldsymbol{\eta}}}
\newcommand{\bom}{{\boldsymbol{\omega}}}
\newcommand{\bOm}{{\boldsymbol{\Omega}}}
\newcommand{\bPhi}{{\boldsymbol{\Phi}}}
\newcommand{\bU}{\mathbf{U}}
\newcommand{\binfty}{\pmb{\infty}}
\newcommand{\BOne}{{\boldsymbol{1}}}
\newcommand{\BTwo}{{\boldsymbol{2}}}
\newcommand{\bsquare}{\pmb{\square}}
\newcommand{\bun}{\mathbf{1}}
\newcommand{\ga}{\mathfrak{a}}
\newcommand{\gothe}{\mathfrak{e}}
\newcommand{\gf}{\mathfrak{f}}
\newcommand{\hgf}{\smash{\widehat{\mathfrak{f}}}}
\newcommand{\gh}{\mathfrak{h}}
\newcommand{\hgh}{\smash{\widehat{\mathfrak{h}}}}
\newcommand{\gm}{\mathfrak{m}}
\newcommand{\gothg}{\mathfrak{g}}
\newcommand{\gotht}{\mathfrak{t}}
\newcommand{\gu}{\mathfrak{u}}
\newcommand{\gA}{\mathfrak{A}}
\newcommand{\gB}{\mathfrak{B}}
\newcommand{\gN}{\mathfrak{N}}
\newcommand{\gM}{\mathfrak{M}}
\newcommand{\veps}{\varepsilon}
\newcommand{\vph}{\varphi}
\newcommand{\dd}{\partial}
\newcommand{\Id}{{\mathrm d}}
\newcommand{\ID}{{\mathrm D}}
\newcommand{\IL}{{\mathrm L}}
\newcommand{\rmi}{{\mathrm i}}
\newcommand{\rP}{{\mathrm P}}
\newcommand{\fnh}{{\text{\textup{FN}}}}
\newcommand{\rme}{{\mathrm{e}}}
\newcommand{\rmN}{{\mathrm{N}}}
\newcommand{\uu}{{\underline{u}}}
\newcommand{\uv}{{\underline{v}}}
\newcommand{\uw}{{\underline{w}}}
\newcommand{\sft}{{\mathsf{t}}}
\newcommand{\vx}{{\vec{\mathrm{x}}}}
\newcommand{\vy}{{\vec{\mathrm{y}}}}
\newcommand{\vz}{{\vec{\mathrm{z}}}}
\newcommand{\bvx}{{\vec{\mathbf{x}}}}
\newcommand{\vt}{{\vec{\mathrm{t}}}}
\newcommand{\vgdd}{{\vec{\mathfrak{d}}}}
\newcommand{\sfE}{\mathsf{E}}

\newcommand{\bbx}{{\boldsymbol{x}}}
\newcommand{\bbu}{{\boldsymbol{u}}}
\newcommand{\bbv}{{\boldsymbol{v}}}
\newcommand{\bbP}{{\boldsymbol{P}}}

\newcommand{\diftat}[2]{ \left. \frac{\Id}{\Id #1} \right|_{#1=#2} }    
\newcommand{\BV}{{\text{\textup{BV}}}}

\DeclareMathOperator{\Span}{span}
\DeclareMathOperator{\Sol}{Sol}
\DeclareMathOperator{\img}{im}
\DeclareMathOperator{\dom}{dom}
\DeclareMathOperator{\id}{id}
\DeclareMathOperator{\rank}{rank}
\DeclareMathOperator{\sym}{sym}
\DeclareMathOperator{\cosym}{cosym}
\DeclareMathOperator{\pt}{pt}
\DeclareMathOperator{\arcsinh}{arcsinh}
\DeclareMathOperator{\coker}{coker}
\DeclareMathOperator{\Hom}{Hom}
\DeclareMathOperator{\End}{End}
\DeclareMathOperator{\Der}{Der}
\DeclareMathOperator{\Mat}{Mat}
\DeclareMathOperator{\poly}{poly}
\DeclareMathOperator{\CDiff}{\mathcal{C}Diff}
\DeclareMathOperator{\Diff}{Diff}
\DeclareMathOperator{\ord}{ord}
\DeclareMathOperator{\volume}{vol}
\DeclareMathOperator{\dvol}{d%\,
	vol}
\DeclareMathOperator{\diag}{diag}
\DeclareMathOperator{\ad}{ad}
\DeclareMathOperator{\Ber}{Ber}
\DeclareMathOperator{\tr}{tr}
\newcommand{\Free}{\text{\textsf{Free}}\,}
\newcommand{\Sl}{\mathfrak{sl}}
\newcommand{\Gl}{\mathfrak{gl}}
\DeclareMathOperator{\const}{const}

\DeclareMathOperator{\Alt}{Alt}

\DeclareMathOperator{\Jac}{Jac}

\DeclareMathOperator{\GH}{gh}
\DeclareMathOperator*{\bigotimesk}{{\bigotimes\nolimits_{\Bbbk}}}
\DeclareMathOperator{\supp}{supp}

\newcommand{\td}{\widetilde{d}}
\newcommand{\tu}{\widetilde{u}}
\newcommand{\tv}{\widetilde{v}}
\newcommand{\tV}{\widetilde{V}}
\newcommand{\hxi}{\widehat{\xi}}
\newcommand{\lshad}{[\![}
\newcommand{\rshad}{]\!]}
\newcommand{\ov}{\overline}
\newcommand{\nC}{{\text{\textup{nC}}}}
\newcommand{\KdV}{{\text{KdV}}}
\newcommand{\ncKdV}{{\text{ncKdV}}}
\newcommand{\mKdV}{{\text{mKdV}}}
\newcommand{\pmKdV}{{\text{pmKdV}}}
\newcommand{\EL}{{\text{EL}}}
\newcommand{\Liou}{{\text{Liou}}}
\newcommand{\scal}{{\text{scal}}} 

\newcommand{\ib}[3]{ \{\!\{ {#1},{#2} \}\!\}_{{#3}} }
\newcommand{\schouten}[1]{\lshad {#1} \rshad}

\newcommand{\by}[1]{\textit{{#1}}}
\newcommand{\jour}[1]{\textit{{#1}}}
\newcommand{\vol}[1]{\textbf{{#1}}}
\newcommand{\book}[1]{\textrm{{#1}}}

\newcommand{\ground}[1]{\text{\textit{\small #1}}}

\newcommand*{\vcenteredhbox}[1]{\begingroup
\setbox0=\hbox{#1}\parbox{\wd0}{\box0}\endgroup}

\DeclareSymbolFont{extraup}{U}{zavm}{m}{n}
\DeclareMathSymbol{\varheartsuit}{\mathalpha}{extraup}{86}
\DeclareMathSymbol{\vardiamondsuit}{\mathalpha}{extraup}{87}

\def\oldvec{\mathaccent "017E\relax } %due to bug in amsmath, need old definition of \vec
\DeclareMathOperator{\Or}{\mathsf{O\oldvec{r}}}

\title[The heptagon\/-\/wheel cocycle in the Kontsevich graph complex]%
{The heptagon\/-\/wheel cocycle\\[1.5pt] in the Kontsevich graph complex}

\author[R.\,Buring]{Ricardo Buring${}^{\text{(\symbol{"61})}}$}
\thanks{${}^{\text{(a)}}$\textit{Address}: %Algebraic Geometry, Topology \& Number Theory,
Institut f\"ur Mathematik, %FB 08 -- Physik, Mathematik und Informatik,
Johannes Gutenberg\/--\/Uni\-ver\-si\-t\"at,
Staudingerweg~9, %4.OG,
\mbox{D-\/55128} Mainz, Germany.}

\author[A.\,V.\,Kiselev]{Arthemy Kiselev${}^{\text{(\symbol{"62},\symbol{"63})}}$}
\thanks{${}^{\text{(b)}}$\textit{Address}: Johann Ber\-nou\-lli Institute for Mathematics and Computer Science, University of Gro\-nin\-gen,
P.O.~Box 407, 9700~AK Groningen, The Netherlands. 
\quad${}^{\text{(c)}}$\textit{E-mail}: \texttt{A.V.Kiselev\symbol{"40}rug.nl}}%,\quad
\author[N.\,J.\,Rutten]{Nina Rutten${}^{\text{(\symbol{"62})}}$}

\dedicatory{Special Issue JNMP 2017 ``Local \textsl{\&} nonlocal symmetries in Mathematical Physics''}

\date{24 November 2017}


\begin{document}
\begin{abstract}
The real vector space of non\/-\/oriented graphs is known to carry a differential graded Lie algebra structure.
%that is 
Cocycles %of the respective bi-grading 
in the Kontsevich graph complex, %i.e. 
expressed using formal sums of %such non\/-\/oriented 
graphs on $n$~vertices and $2n-2$~edges, induce --\,under the orientation mapping\,-- infinitesimal symmetries of classical Poisson structures on arbitrary finite-dimensional affine real manifolds.
%Recently % Willwacher in 2010, GRT v1
Willwacher % with Dolgushev and Rogers % in 2010 without
has % RB: to replace 'Recently'
stated the existence of a nontrivial
cocycle that contains the $(2\ell+1)$-\/wheel graph with a nonzero coefficient
at every $\ell\in\BBN$.
%%%
We present detailed calculations of the differential of graphs;
for the tetrahedron and pentagon\/-\/wheel cocycles, 
consisting at $\ell = 1$ and~$\ell = 2$ of one and two graphs respectively,
the cocycle condition $\Id(\gamma) = 0$ is verified by hand.
%%%
%which can be represented by using one and two connected graphs respectively, were the previously known % RB: nontrivial
%solutions of the %cohomological % RB: cohomological equation is weird
%equation $\Id(\gamma) = 0$ at $\ell = 1$ and~$\ell = 2$. 
%Now we describe % obtain, have found
For the next, heptagon\/-\/wheel cocycle (known to exist at~$\ell = 3$),
we provide an explicit representative: it consists of $46$~graphs on $8$~vertices and $14$~edges.
%%%
%We establish % RB: assert is vague; do we prove or not
%that this solution is unique in the respective vertex\/-\/edge bi\/-\/grading. % class
\end{abstract}
\maketitle

\thispagestyle{empty}
\enlargethispage{0.7\baselineskip}
\subsection*{Introduction}
The structure of differential graded Lie algebra on the space of non-oriented graphs, as well as the cohomology groups of the graph complex, were introduced by Kontsevich in the context of mirror symmetry \cite{MKParisECM, MKZurichICM}.
It can be shown that by orienting a graph cocycle on $n$ vertices and $2n-2$ edges (and by adding to every graph in that cocycle two new edges going to two sink vertices) in all such ways that each of the $n$ old vertices is a tail of exactly two arrows, and by placing a copy of a given Poisson bracket $\cP$ in every such vertex, one obtains an infinitesimal symmetry of the space of Poisson structures. This construction is universal with respect to all finite-dimensional affine real manifolds (see \cite{Ascona96} and \cite{f16}).\footnote{%
%%%
The dilation $\dot{\cP} = \cP$, also universal with respect to all Poisson manifolds, is obtained by orienting the graph~$\bullet$ on one vertex and no edges, yet that graph is not a cocycle, $\Id(\bullet)=-\bullet\!\!{-}\!\!\bullet\neq0$. The single\/-\/edge graph $\bullet\!\!{-}\!\!\bullet \in\ker\Id$ on two vertices is a cocycle but its bi\/-\/grading differs from~$(n,2n-2)$. However, by satisfying the zero\/-\/curvature equation $\Id(\bullet\!\!{-}\!\!\bullet)+\tfrac{1}{2}[\bullet\!\!{-}\!\!\bullet,\bullet\!\!{-}\!\!\bullet]=0$ the graph $\bullet\!\!{-}\!\!\bullet$ is a Maurer\/--\/Cartan element in the graph complex.}
%%%
Until recently two such %other 
differential\/-\/polynomial symmetry flows were known (of nonlinearity degrees $4$ and $6$ respectively).
Namely, the tetrahedral graph flow $\dot{\cP} = \cQ_{1:\frac{6}{2}}(\cP)$ was proposed in the seminal paper~\cite{Ascona96} (see also~\cite{f16, tetra16}).
Consisting of $91$~oriented bi-vector graphs on $5+1 = 6$ vertices, the Kontsevich\/--\/Willwacher pentagon\/-\/wheel flow will presently be described in~\cite{WeFactorize5Wheel}.

The cohomology of the graph complex in degree~$0$ is known to be isomorphic to
the Grothendieck\/--\/Teichm\"uller %(GRT) 
Lie algebra~$\mathfrak{grt}$ 
(see~\cite{DrinfeldGRT1990} and~\cite{WillwacherGRT}); 
under the isomorphism, the~$\mathfrak{grt}$ %GRT Lie algebra 
generators correspond to %The space of 
nontrivial cocycles. 
%in the non-oriented graph complex is also known to be isomorphic to the vector space of generators of the of 
% RB: ??? vector space of degree-0 cocycles is isomorphic to grt Lie algebra
Using this correspondence, Willwacher 
  %together with Dolgushev and Rogers %RB: Dolgushev maybe, Rogers probably not
gave %predicts 
in~\cite[Proposition 9.1]{WillwacherGRT} %arXiv:pp.28--29
the existence proof for an infinite sequence of the Deligne\/--\/Drinfel'd nontrivial cocycles 
on $n$~vertices and $2n-2$~edges.
%in the non\/-\/oriented graph complex. % RB: predicts sounds uncertain/conjectural, but here it is not
(%Several 
Formulas %expressing 
which describe these cocycles in terms of the $\mathfrak{grt}$ Lie algebra generators are given in %\S6
the preprint~\cite{WillwacherRossi14042047}.) 
To be specific, at each $\ell \in \BBN$ every cocycle from that sequence contains the $(2\ell+1)$-\/wheel with nonzero coefficient (e.g., the tetrahedron alone making the cocycle~$\boldsymbol{\gamma}_3$ at~$\ell=1$), and possibly other graphs on $2\ell+2$~vertices and $4\ell+2$~edges.
For instance, at~$\ell=2$ the pentagon\/-\/wheel cocycle~$\boldsymbol{\gamma}_5$ consists of two graphs, see Fig.~\ref{FigPentaCocycle} on p.~\pageref{FigPentaCocycle} below.

In this paper we describe %find 
the next one, the heptagon\/-\/wheel cocycle~$\boldsymbol{\gamma}_7$ from that sequence of solutions to the equation 
\[% RB: please better notation
\Id\Bigl({\sum_{\{\text{graphs}\}}} (\text{coefficient} \in \BBR)\cdot(\text{graph with an ordering of its edge set})\Bigr) = 0.
\]
Our representative of the %The new 
cocycle~$\boldsymbol{\gamma}_7$ consists of $46$~connected graphs on $8$~vertices and $14$~edges.
(This number of nonzero coefficients can be increased by adding a coboundary.) % RB: maybe also decreased? I have not checked
This solution has been obtained straightforwardly, that is, by solving the graph equation $\Id(\boldsymbol{\gamma}_7) = 0$ directly.
One could try reconstructing the cocycle~$\boldsymbol{\gamma}_7$ from a set of the $\mathfrak{grt}$ Lie algebra generators, which are known in low degrees.
  %\marginpar{Ref} % RB: TODO: find generators \sigma
Still an explicit verification that $\boldsymbol{\gamma}_7 \in \ker \Id$ would be 
appropriate %advisable %compulsory 
for that way of reasoning.
%\marginpar{Fix} % RB: it is given by Willwacher's theorem/algorithm/proof

%\enlargethispage{0.7\baselineskip}
In this paper we also confirm 
that the three cocycles known so far --\,namely the tetrahedron and pentagon-{} and heptagon\/-\/wheel solutions\,-- span the space of
   %exhaust the set of representatives for 
nontrivial cohomology classes which are built of connected graphs on $n \leqslant 8$ vertices and $2n-2$~edges.
At~$n=9$, there is a unique nontrivial cohomology class with graphs on nine vertices and sixteen edges: namely, the Lie bracket $[\boldsymbol{\gamma}_3,\boldsymbol{\gamma}_5]$ of the previously found cocycles.
(Brown showed in~\cite{Brown} that the elements~$\sigma_{2\ell+1}$ in the Lie algebra~$\mathfrak{grt}$ which --\,under the Willwacher isomorphism\,-- correspond to the wheel cocycles~$\boldsymbol{\gamma}_{2\ell+1}$ generate a free Lie algebra; hence it was expected that the cocycle $[\boldsymbol{\gamma}_3,\boldsymbol{\gamma}_5]$ is non\/-\/trivial.)
%%%
%%%  -- Authors' comment --
%In the above, what was not pronounced explicitly: "the Lie bracket in grt is mapped to the Lie bracket in the graph complex cohomology;  non-zero element in grt is always mapped to a non-"zero" graph sum which, on top of that, is cohomologically nontrivial."
%%%  ----------------------
%%%
%This is remarkable because the commutator $[\cQ_{1:\frac{6}{2}}(\cP), \Or(\boldsymbol{\gamma}_5)(\cP)]$ of the Kontsevich tetrahedral and Kontsevich\/--\/Willwacher pentagon\/-\/wheel flows on the spaces of Poisson bi-vectors~$\cP$, were that commutator nontrivial with respect to the usual Poisson differential
   %\marginpar{Fix} % RB: for each Poisson structure there is a Poisson differential
%and modulo improper terms (that vanish identically or by virtue of the Jacobi identity for~$\cP$), would stem from a connected ($n=9$)-\/vertex graph under the orientation mapping~$\Or$.
%%%
To verify that the list of currently known $\Id$-\/cocycles is exhaustive --\,under all the assumptions which were made about the graphs at our disposal\,-- at every $n \leqslant 9$ we count the dimension of the space of cocycles minus the dimension of the space of respective coboundaries.\footnote{The proof scheme is computer\/-\/assisted (cf.~\cite{f16,cpp});
it can be applied to the study of other cocycles: either on higher number of vertices or built at arbitrary $n \geqslant 2$ from not necessarily connected graphs%
%(see Remark~\ref{RemWaterfallGenerator} on p.~\pageref{RemWaterfallGenerator} below)
.}
%%%
Our findings fully %completely 
match the dimensions %estimates 
from~\cite[Table~1]{KhoroshkinWillwacherZivkovic}.
%; this also adds to the work~\cite{WillwacherZivkovic2015Table} in which overall numbers of graphs, satisfying all the assumptions but not necessarily showing up in any cocycles, were calculated.

This text is structured as follows.
Necessary definitions and some notation from the graph complex theory are recalled in~\S\ref{SecGraphComplex}.
These notions are illustrated in~\S\ref{SecExamples} where a step-by-step calculation of the (vanishing) differentials $\Id(\boldsymbol{\gamma}_3)$ and $\Id(\boldsymbol{\gamma}_5)$ is explained.
%Summarized in Theorems~\ref{ThHepta} and~\ref{ThCount} in \S\ref{SecResult}, our main results are the heptagon-wheel %nontrivial 
%solution of the cohomological equation $\Id(\boldsymbol{\gamma}_7) = 0$ and, respectively, the count of number of cocycles modulo coboundaries
Our main result is Theorem~\ref{ThHepta} 
with the heptagon\/-\/wheel solution of the %cohomological
equation $\Id(\boldsymbol{\gamma}_7) = 0$.
Also in~\S\ref{SecResult}, in Proposition~\ref{ThCount} we verify 
the count of number of cocycles modulo coboundaries which are
formed by all connected graphs on $n$~vertices and $2n-2$~edges (here $4 \leqslant n \leqslant  9$). % RB: "cohomological equation" is weird
The graphs which constitute% the solution
~$\boldsymbol{\gamma}_7$ are drawn on pp.~\pageref{pStartHepta}--\pageref{pEndHepta} in Appendix~\ref{AppHepta}.
The code in \textsc{Sage} programming language, allowing one to calculate the differential for a given graph~${\gamma}$ and ordering~$E(\gamma)$ on the set of its edges, is contained in Appendix~\ref{AppSage}; the same code can be run to calculate the dimension of graph cohomology groups.

The main purpose of this paper is to provide a pedagogical introduction into the subject.%
\footnote{The first %Another 
example of practical calculations of the graph cohomology --\,with respect to the %other, 
edge contracting differential\,-- is found in~\cite{BarNatan};
a wide range of vertex\/-\/edge bi\/-\/degrees is considered there.%
}
Besides, the formulas of the three cocycle representatives will be helpful in the future search of an easy recipe to calculate all the wheel cocycles~$\boldsymbol{\gamma}_{2\ell+1}$. (No general recipe is known yet, except for a longer reconstruction of those cohomology group elements from the generators of Lie algebra~$\mathfrak{grt}$.) Thirdly, our present knowledge of both the cocycles~$\boldsymbol{\gamma}_i$ and the respective flows $\dot{\cP}=\cQ_i(\cP)$ on the spaces of Poisson structures will be important for testing and verifying explicit formulas of the orientation mapping~$\Or$ such that $\cQ_i=\Or(\boldsymbol{\gamma}_i)$.

\section{The non\/-\/oriented graph complex}\label{SecGraphComplex}
\noindent%
We work with the real vector space %freely 
generated by finite non\/-\/oriented graphs\footnote{The vector space of graphs under study is infinite dimensional; however, it is endowed with the bi\/-\/grading ($\#$vertices, $\#$edges) so that all the homogeneous components are finite dimensional.} 
%%%  %$k$-\/vertex  graphs, 
without multiple edges nor tadpoles and endowed with a wedge ordering of edges:
%, e.g., $E=e_1\wedge\dots\wedge e_{2k-2}$; 
by definition,
%We postulate that 
an edge swap $e_i\wedge e_j = -e_j\wedge e_i$ 
implies the change of sign in front of the graph at hand.
Topologically equal graphs are equal as vector space elements if their edge orderings~$\sfE$ differ by an even permutation; otherwise, the graphs are opposite to each other (i.e.\ they differ by the factor~$-1$).

\begin{define}\label{DefZeroGraph}%\footnote{
%The edges are antipermutable so that 
A graph which equals minus itself 
--\,under a symmetry that induces a parity\/-\/odd permutation of edges\,--
is %proclaimed to be equal to 
called a \emph{zero graph}.
In particular (view $\bullet\!\!{-}\!\!\bullet\!\!{-}\!\!\bullet$), every graph possessing a symmetry which swaps an odd number of edge pairs is a zero graph. 
\end{define}

\begin{notation}
For a given labelling of vertices in a graph, we denote by $ij$ (equivalently, by~$ji$)
the edge connecting the vertices~$i$ and~$j$. For instance, both $\mathsf{12}$ and~$\mathsf{21}$  is the notation for the edge between the vertices~$\mathsf{1}$ and~$\mathsf{2}$. (No multiple edges are allowed, hence~$\mathsf{12}$ is \emph{the} edge. Indeed, by Definition~\ref{DefZeroGraph} all graphs with multiple edges would be %are 
zero graphs.)
We also denote by $N(v)$ the valency of a vertex~$v$.
\end{notation}

\begin{example}
The $4$-\/wheel $\mathsf{12}\wedge \mathsf{13}\wedge \mathsf{14}\wedge \mathsf{15}\wedge \mathsf{23}\wedge \mathsf{25}\wedge \mathsf{34}\wedge \mathsf{45} = I\wedge\cdots\wedge VIII$ or likewise, the $2\ell$-\/wheel at any $\ell>1$ is a zero graph; here, the %mirror 
reflection symmetry is $I\rightleftarrows III$,\ $V\rightleftarrows VII$, 
and~$VI \rightleftarrows VIII$.
\end{example}

Note that every %connected 
%component of 
term in a sum of non\/-\/oriented graphs~$\gamma$ with real coefficients is fully encoded by an ordering~$\sfE$ on the set of adjacency relations for its \mbox{vertices~$v$ (if~$N(v)>0$).}
From now on, we assume $N(v)\geqslant3$ %for vertices of all graphs
unless stated otherwise explicitly.

\begin{example}\label{Ex3Wheel}
The tetrahedron (or $3$-\/wheel) is the full %non-oriented  
graph on four vertices and six edges (enumerated in the ascending
order: $\mathsf{12}=I$, $\ldots$, $\mathsf{34}=VI$),
\[
\phantom{ml}\boldsymbol{\gamma}_3 = %\pm 
\mathsf{12}%^{(I)}
\wedge \mathsf{13}%^{(II)}
\wedge \mathsf{14}%^{(III)}
\wedge \mathsf{23}%^{(IV)}
\wedge \mathsf{24}%^{(V)}
\wedge \mathsf{34}%^{(VI)} 
=
I\wedge\cdots\wedge VI
=
\phantom{+}
%%% Figure:
\raisebox{4pt}[1pt][1pt]{%
%%% The 3-wheel.
{\unitlength=0.07mm
\begin{picture}(0,0)(0,60)
\qbezier%[40]
(0,0)(150,0)(150,0)
\qbezier%[40]
(0,0)(75,130)(75,130)
\qbezier%[40]
(150,0)(75,130)(75,130)
\put(0,0){\circle*{5}}
\put(150,0){\circle*{5}}
\put(75,130){\circle*{5}}
%%%
\put(75,43.33){\circle*{5}}
\qbezier%[40]
(0,0)(75,43.33)(75,43.33)
\qbezier%[40]
(150,0)(75,43.33)(75,43.33)
\qbezier%[40]
(75,130)(75,43.33)(75,43.33)
%%%
\put(-4.5,-1){\llap{{\tiny{\textsf{1}}}}}
\put(79.5,129){{{\tiny{\textsf{2}}}}}
\put(155,-1){{{\tiny{\textsf{3}}}}}
\end{picture}
}
}
%\phantom{mmm}.
\]
This graph is nonzero. (The axis vertex is labelled~\textsf{4} in this figure.) 
\end{example}

\begin{example}\label{Ex5Wheel}
The linear combination~$\boldsymbol{\gamma}_5$ of two 6-vertex 10-edge graphs,
namely, of the pentagon wheel and triangular prism with one extra diagonal
(here, $\mathsf{12}=I$ and so~on),
\begin{multline*}
\boldsymbol{\gamma}_5 = 
\mathsf{12}%^{(I)}
\wedge \mathsf{23}%^{(II)}
\wedge \mathsf{34}%^{(III)}
\wedge \mathsf{45}%^{(IV)}
\wedge \mathsf{51}%^{(V)} 
\wedge \mathsf{16}%^{(VI)} 
\wedge \mathsf{26}%^{(VII)}
\wedge \mathsf{36}%^{(VIII)}
\wedge \mathsf{46}%^{(IX)}
\wedge \mathsf{56}%^{(X)}
\\
 + \tfrac{5}{2}\cdot 
\mathsf{12}%^{(I)}
\wedge \mathsf{23}%^{(II)}
\wedge \mathsf{34}%^{(III)} 
\wedge \mathsf{41}%^{(IV)}
\wedge \mathsf{45}%^{(V)}
\wedge \mathsf{15}%^{(VI)}
\wedge \mathsf{56}%^{(VII)}
\wedge \mathsf{36}%^{(VIII)}
\wedge \mathsf{26}%^{(IX)}
\wedge \mathsf{13}%^{(X)},
\end{multline*}
%\begin{align*}
%\gamma_5 =\quad& \mathsf{12}^{I}\wedge \mathsf{23}^{II}\wedge \mathsf{34}^{III}\wedge \mathsf{45}^{IV} \wedge \mathsf{51}^{V} \wedge \mathsf{16}^{VI} \wedge \mathsf{26}^{VII}\wedge \mathsf{36}^{VIII}\wedge \mathsf{46}^{IX}\wedge \mathsf{56}^{X}\\
 %&+ \frac{5}{2}\cdot \mathsf{12}^{I}\wedge \mathsf{23}^{II}\wedge \mathsf{34}^{III} \wedge \mathsf{41}^{IV}\wedge \mathsf{45}^{V}\wedge \mathsf{15}^{VI}\wedge \mathsf{56}^{VII}\wedge \mathsf{36}^{VIII}\wedge \mathsf{26}^{IX}\wedge \mathsf{13}^{X},
%\end{align*}
%RB: <Graph complex: differential(SAGE) p.1-2;15.07.17>
is drawn in Fig.~\ref{FigPentaCocycle} on p.~\pageref{FigPentaCocycle}
   %Theorem~\ref{ThMainMany} on p.~\pageref{Where5WheelIs} %{ThMainMany} 
below (cf.~\cite{BarNatan}).
\end{example}

Let $\gamma_1$ and $\gamma_2$ be connected non-oriented graphs. The definition of insertion $\gamma_1\circ_i\gamma_2$ %\mathbin{\circ_i}
of the entire graph $\gamma_1$ into vertices of $\gamma_2$ and the construction of Lie bracket $[\cdot,\cdot]$ of graphs and differential~$\Id$ in the non\/-\/oriented graph complex, referring to a sign convention, are as follows
(cf.~\cite{Ascona96} and~\cite{DolgushevRogersWillwacher,KhoroshkinWillwacherZivkovic,WillwacherGRT}); these definitions apply to sums of graphs by linearity.

\begin{define}\label{DefInsert}
The insertion $\gamma_1\circ_i\gamma_2$ of an $n_1$-vertex graph $\gamma_1$ with ordered set of edges $\sfE(\gamma_1)$ into a graph $\gamma_2$ with $\#\sfE(\gamma_2)$ edges on $n_2$ vertices is a %the 
sum of %$n_2$ 
graphs on $n_1+n_2-1$ vertices and $\#\sfE(\gamma_1)+\#\sfE(\gamma_2)$ edges. 
Topologically, the sum $\gamma_1\circ_i\gamma_2 = \sum(\gamma_1\rightarrow v \text{ in }\gamma_2)$ %\in V(\gamma_2)
consists of all the graphs in which a vertex $v$ from $\gamma_2$ is replaced by the entire graph $\gamma_1$ and the edges touching $v$ in $\gamma_2$ are re-attached to the vertices of $\gamma_1$ in all possible ways.\footnote{Let the enumeration of vertices in every such term in the sum start running over the enumerated vertices in $\gamma_2$ until $v$ is reached.
Now the enumeration counts the vertices in the graph $\gamma_1$ and then it resumes with the remaining vertices (if any) that go after~$v$ %proceed %antecede
in~$\gamma_2$.}
By convention, in every new term the edge ordering is~$\sfE(\gamma_1)\wedge \sfE(\gamma_2)$. \end{define}
%Ricardo, we can have it vice versa as in SAGE: extra edge last, if you wish.
%See\footnote below% = page -4up- %

To simplify sums of graphs, first eliminate the zero graphs.
Now suppose that in a sum, two non\/-\/oriented graphs, say $\alpha$ and $\beta$, are isomorphic (topologically, i.e.\ regardless of the respective vertex labellings and edge orderings $\sfE(\alpha)$ and~$\sfE(\beta)$).
By using that isomorphism, which establishes a 1--1~correspondence between the edges, extract the sign from the equation $\sfE(\alpha) = \pm \sfE(\beta)$.
If~``$+$'', then $\alpha = \beta$; else $\alpha = -\beta$.
Collecting similar terms is now elementary.

% "Our" definition is ugly for \cdot\circ_i\(\cdot\cdot) without edges.
% Enumeration of vertices does play its role to encode a graph but it is irrelevant for determining the sign \pm of any term!
% Cancellations are possible only between the graphs which were produced from different pairs of arguments (\gamma_1,\gamma_2); in one sum \gamma_1\circ_i\gamma_2 all topologically equal terms are produced with a common sign (yet these newly produced terms can each intrinsically cancel out whenever this is a zero graph).
\begin{lemma}
The bi\/-\/linear graded skew\/-\/symmetric operation,
\[%\begin{equation}
[\gamma_1,\gamma_2] = \gamma_1\circ_i\gamma_2 - (-)^{\#\sfE(\gamma_1)\cdot\#\sfE(\gamma_2)} \gamma_2\circ_i\gamma_1,
\]%\end{equation}
is a Lie bracket on the vector space $\mathfrak{G}$ of non-oriented graphs.\footnote{The postulated precedence or antecedence of the wedge product of edges from $\gamma_1$ with respect to the edges from $\gamma_2$ in every graph within $\gamma_1\circ_i\gamma_2$ produce the operations $\circ_i$ which coincide with or, respectively, differ from Definition~\ref{DefInsert} by the sign factor $(-)^{\#\sfE(\gamma_1)\cdot\#\sfE(\gamma_2)}$. The same applies to the Lie bracket of graphs $[\gamma_1,\gamma_2]$ if the operation $\gamma_1\circ_i\gamma_2$ is %denotes
the insertion of $\gamma_2$ into $\gamma_1$ (as in~\cite{KhoroshkinWillwacherZivkovic}). Anyway, the notion of $\Id$-cocycles which we presently recall is well defined and insensitive to such sign~ambiguity.}
\end{lemma}

\begin{lemma}
The operator $\Id(\text{graph}) = [\bullet\!\!\!-\!\!\!\bullet, \text{graph}]$ is a differential: $\Id^2 = 0$.
\end{lemma}

In effect, the mapping~$\Id$ blows up every vertex~$v$ in its argument in such a way that whenever the number of adjacent vertices $N(v) \geqslant 2$ is sufficient, each end of the inserted edge $\bullet\!\!{-}\!\!\bullet$ is connected with the rest of the graph by at least one~edge.

\begin{theor}[\cite{Ascona96}] % = Summarising,
The real vector space $\mathfrak{G}$ of non-oriented graphs is a differential graded Lie algebra \textup{(}dgLa\textup{)} with % respect to the
Lie bracket $[\cdot,\cdot]$ and differential $\Id = [\bullet\!\!{-}\!\!\bullet, \cdot]$.
The differential~$\Id$ is a graded derivation of the bra\-cket~$[\cdot,\cdot]$
\textup{(}due to the Jacobi identity for this Lie algebra structure\textup{)}.
\end{theor}

The graphs~$\boldsymbol{\gamma}_3$ and~$\boldsymbol{\gamma}_5$ from Examples~\ref{Ex3Wheel}
and~\ref{Ex5Wheel} are $\Id$-\/cocycles (this will be shown in~\S\ref{SecExamples}).
Therefore, their commutator $[\boldsymbol{\gamma}_3,\boldsymbol{\gamma}_5]$ is also in~$\ker\Id$.
% RB: either Examples in section ... or number examples by section.
Neither $\boldsymbol{\gamma}_3$ nor~$\boldsymbol{\gamma}_5$ is exact, hence marking a nontrivial cohomology class in the non\/-\/oriented graph complex.

\begin{theor}[{{\cite[Th.\:5.5]{DolgushevRogersWillwacher}}}]\label{ThWillwacherWheels}
  %\label{ThWheelCocycles}
At every $\ell \in \mathbb{N}$ in the connected graph complex there is a nontrivial
$\Id$-cocycle on $2\ell + 1$ vertices and $4\ell+2$ edges.
Such cocycle contains the $(2\ell+1)$-\/wheel in which, by definition, the axis vertex is connected with every other vertex by a spoke so that each of those $2\ell$ vertices is adjacent to the axis and two neighbours\textup{;} the cocycle marked by the $(2\ell+1)$-\/wheel graph can contain other $(2\ell+1, 4\ell+2)$-\/graphs.%\textup{(}see Example~\textup{\ref{Ex5Wheel})}.
\end{theor}
% RB: be more precise: the wheel has a nontrivial coefficient in the cocycle

\begin{example}
For $\ell=3$ the heptagon wheel cocycle $\boldsymbol{\gamma}_7$, which we present in this paper, consists of the heptagon\/-\/wheel graph on $(2\cdot 3 + 1)+1 = 8$ vertices and $2(2\cdot 3 + 1) = 14$ edges and forty\/-\/five other graphs with equally many vertices and edges (hence of the same number of generators of their homotopy groups, or basic loops: $7 = 14 - (8-1)$), and with real coefficients.
All these weighted graphs are drawn in Appendix~\ref{AppHepta} (see pp.~\pageref{pStartHepta}--\pageref{pEndHepta}).
The chosen --\,lexicographic\,-- ordering of edges in each term %connected component 
is read from the encoding of every such graph (see also Table~\ref{TabHeptagonCocycle} on p.~\pageref{TabHeptagonCocycle}; %\marginpar{Don't need table for edge ordering} 
% RB: the first edge is the one from vertex 1 to the next-lowest vertex, etc. it's a simple convention.
each entry of that table is a listing $I \prec \cdots \prec XIV$ of the ordered edge set, followed by the coefficient of that graph). % RB: Far-away explanation of table
A verification of the cocycle condition $\Id(\boldsymbol{\gamma}_7) = 0$ for this solution is computer-assisted; it has been performed by using the code (in {\sc Sage} programming language) which is contained in Appendix~\ref{AppSage}.% on p.~\pageref{AppSage}.
\end{example}

\section{Calculating the differential of graphs}\label{SecExamples}
\begin{example}[$\Id\boldsymbol{\gamma}_3=0$]\label{ExDTetra}
The tetrahedron $\boldsymbol{\gamma}_3$ is the full graph on $n=4$ vertices; we are free to choose any ordering of the six edges in it, 
%NJR: more explanation here or earlier on choosing an ordering
so let it be lexicographic: \[\sfE(\boldsymbol{\gamma}_3) = \mathsf{12}\wedge\mathsf{13}\wedge\mathsf{14}\wedge\mathsf{23}\wedge\mathsf{24}\wedge\mathsf{34} = I\wedge II\wedge III\wedge IV \wedge V\wedge VI.\] 
The differential of this graph is equal to %We calculate its differential,
\[
\Id(\boldsymbol{\gamma}_3) = [\bullet\!\!{-}\!\!\bullet ,\boldsymbol{\gamma}_3] 
= \bullet\!\!{-}\!\!\bullet \circ_i\gamma_3 
  - (-)^{\# E(\bullet\!\!{-}\!\!\bullet)\cdot\# E(\boldsymbol{\gamma}_3)} \boldsymbol{\gamma}_3\circ_i\bullet\!\!{-}\!\!\bullet
=\bullet\!\!{-}\!\!\bullet \circ_i\gamma_3 
  - \boldsymbol{\gamma}_3\circ_i\bullet\!\!{-}\!\!\bullet,
\]
since $\#\sfE(\boldsymbol{\gamma}_3)=6$. Note that every vertex of valency one appears twice in $\Id(\boldsymbol{\gamma}_3)$: namely in the minuend (where the edge ordering is $E\wedge I\wedge\cdots\wedge VI$ by definition of~$\circ_i$) and subtrahend (where the edge ordering is $ I\wedge\cdots\wedge VI\wedge E$). Because these edge orderings differ by a parity-even permutation, such graphs in $\bullet\!\!{-}\!\!\bullet \circ_i\boldsymbol{\gamma}_3$ and $\boldsymbol{\gamma}_3 \circ_i \bullet\!\!{-}\!\!\bullet$ carry the same sign. Hence they cancel in the difference $\bullet\!\!{-}\!\!\bullet \circ_i\boldsymbol{\gamma}_3 - \boldsymbol{\gamma}_3 \circ_i \bullet\!\!{-}\!\!\bullet$, and no longer shall we pay any attention to the leaves, absent in the differential of any graph. 
It is readily seen that the twenty\/-\/four graphs
($24 = 4\text{ vertices} \cdot \binom{3}{1} \cdot 2 \text{ ends of } \bullet\!\!{-}\!\!\bullet$) we are left with in~$\Id(\boldsymbol{\gamma}_3)$ are of the shape drawn here.
   %\marginpar{To do picture tetra blow up}\\
\begin{figure}[htb]
\[
\raisebox{3pt}[1pt][1pt]{
%%% The d(tetra) (p.7); upper-left edge blow, with edge'-edge'' inscription.
{\unitlength=0.09mm
\begin{picture}(0,0)(0,60)
\put(0,0){\line(0,1){85}}
%\put(75,130){\line(-5,-4){75}}
\qbezier[60](75,130)(0,85)(0,85)
\put(0,85){\circle*{10}}
%%%
\put(0,0){\line(1,0){150}}
%\qbezier[60](0,0)(150,0)(150,0)
%%%
%\qbezier[60](0,0)(75,130)(75,130)
\qbezier[60](150,0)(75,130)(75,130)
\put(0,0){\circle*{10}}
\put(150,0){\circle*{10}}
\put(75,130){\circle*{10}}
%%%
\put(75,43.33){\circle*{10}}
\qbezier[60](0,0)(75,43.33)(75,43.33)
\qbezier[60](150,0)(75,43.33)(75,43.33)
\qbezier[60](75,130)(75,43.33)(75,43.33)
%%%
\put(-6,-2){\llap{{\tiny$v_i$}}}
\put(81,132){{\tiny$v_j$}}
\put(-3,40){\llap{{\tiny edge$'$}}}
\put(45,113){\llap{{\tiny edge$''$}}}
\end{picture}
}
}
\phantom{mmml}
=\phantom{l}
\raisebox{7pt}[1pt][1pt]{
%%% The d(tetra) (p.7); vertical symmetry line (not drawn).
{\unitlength=0.07mm
\begin{picture}(0,0)(0,60)
%\qbezier[40](0,0)(150,0)(150,0)
\qbezier[60](0,0)(75,130)(75,130)
\qbezier[60](150,0)(75,130)(75,130)
\put(0,0){\circle*{10}}
\put(150,0){\circle*{10}}
\put(75,130){\circle*{10}}
%%%
\put(75,43.33){\circle*{10}}
\qbezier[60](0,0)(75,43.33)(75,43.33)
\qbezier[60](150,0)(75,43.33)(75,43.33)
\qbezier[60](75,130)(75,43.33)(75,43.33)
%%%
\qbezier[60](0,0)(75,-43.33)(75,-43.33)
\qbezier[60](150,0)(75,-43.33)(75,-43.33)
\put(75,-43.33){\circle*{10}}
\end{picture}
}
}
\phantom{mmmml}\:
%=\ ?
\qquad{}\lefteqn{\text{(see Remark~\ref{RemIncidentTetra})}}
\]
\end{figure}
A vertex is blown up to the new edge $E = \bullet\!\!{-}\!\!\bullet$ whose ends are both attached to the rest of the graph along the old edges. This shape can be obtained in two ways: by blowing up~$v_i$, so that edge$'$ is the newly inserted edge, or by blowing up~$v_j$, so that edge$''$ is the newly inserted edge. By Lemma~\ref{LemmaHandshakesCancel} below we conclude that~$\Id(\boldsymbol{\gamma}_3)=0$.
\end{example}

\begin{rem}\label{RemIncidentTetra}
Incidentally, every graph which was obtained in~$\Id(\boldsymbol{\gamma}_3)$ 
   %\marginpar{To do picture symmetric tetra} 
itself is a zero graph. Indeed, it is symmetric with respect to a flip over the vertical line and this symmetry swaps three edge pairs (see Definition~\ref{DefZeroGraph}).
\end{rem}

\begin{lemma}[handshake]\label{LemmaHandshakesCancel}
In the differential of any graph~$\gamma$ %not necessarily connected
such that the valency of all vertices in~$\gamma$ is strictly greater than two, 
the graphs in which one end of the newly inserted edge $\bullet\!\!{-}\!\!\bullet$ has valency two, all cancel.
\end{lemma}

\begin{proof}
Let $v$ be such a vertex in $\Id(\gamma)$, i.e.\ %that is, 
the vertex~$v$ is an %the 
end of the inserted edge $\bullet\!\!{-}\!\!\bullet$ and it has valency~2. Locally (near~$v$), we have either $_a\!\bullet\!\!\!\frac{\,\,\, E'}{}\!\!\!\bullet_v\!\!\!\!\!\frac{\,\,\,\text{Old}'}{}\!\bullet_b$ or $_a\!\bullet\!\!\!\frac{\,\,\,\text{Old}''}{}\!\!\!\bullet_v\!\!\!\!\!\frac{\,\,\, E''}{}\!\bullet_b$. In the two respective graphs in~$\Id(\gamma)$ the rest, consisting only of old edges and vertices of valency~$\geqslant3$
from $\boldsymbol{\gamma}$, is the same. Yet the two graphs are topologically equal; furthermore, they have the same ordering of edges except for $E'=\text{Old}''$ and $\text{Old}'=E''$. Recall that by construction, the edge ordering of the first graph is $E'\wedge \cdots\wedge \text{Old}'\wedge\cdots$, whereas for the second graph it is $E''\wedge\cdots\wedge \text{Old}''\wedge\cdots$; the new edge always goes first. So effectively, two edges are swapped. Therefore,
\[
E''\wedge\cdots\wedge \text{Old}''\wedge\cdots = \text{Old}'\wedge\cdots\wedge E'\wedge\cdots\\
= -E'\wedge\cdots\wedge \text{Old}'\wedge\cdots.
\]
Hence in every such pair in~$\Id(\gamma)$, the graphs occur with opposite signs. 
Moreover, the initial hypothesis $N(a)\geqslant3$ about the valency of all vertices~$a$ in the graph~$\gamma$ guarantees that the cancelling pairs of graphs in~$\Id(\gamma)$ do not intersect,\footnote{%
This is why the assumption $N(v)\geqslant3$ is important. 
Indeed, %for a would\/-\/be vertex~$v$ of valency two in~$\gamma$, 
the disjoint\/-\/pair cancellation mechanism does work only for chains with even numbers of valency\/-\/two vertices~$v$ in~$\gamma$. Here is an example (of one such vertex~$v$ between~$a$ and~$b$) when it actually does not: in the differential of a graph that contains
$_a\!\bullet\!\!\!\frac{\,\,\, I\,\,}{}\!\!\!\bullet_v\!\!\!\!\!\frac{\,\,\,II\,\,}{}\!\bullet_b$,
we locally obtain 
$_a\!\bullet\!\!\!\frac{\,\,\, E\,\,}{}\!\!\!\bullet\lefteqn{{}_{\!a'}} \!\!\!\frac{\,\,\, I\,\,}{}\!\!\!\bullet\lefteqn{{}_{\!v}}\!\!\!\!\frac{\,\,\,II\,\,}{}\!\!\!\bullet_b +
%%%
_a\!\bullet\!\!\!\frac{\,\,\, I\,\,}{}\!\!\!\bullet\lefteqn{{}_{\!v}} \!\!\!\frac{\,\,\, E\,\,}{}\!\!\!\bullet\lefteqn{{}_{\!v'}}\!\!\!\frac{\,\,\,II\,\,}{}\!\!\!\bullet_b +
%%%
_a\!\bullet\!\!\!\frac{\,\,\, I\,\,}{}\!\!\!\bullet\lefteqn{{}_{\!v}} \!\!\!\frac{\,\,\, II\,\,}{}\!\!\!\bullet\lefteqn{{}_{\!\mathstrut b'}}\!\!\!\frac{\,\,\,E\,\,}{}\!\!\!\bullet_b$, so that the middle term can be cancelled against either the first or the last one but not with both of them simultaneously.}
%%%
and thus all cancel.
\end{proof}

\begin{cor}[to Lemma~\ref{LemmaHandshakesCancel}]\label{CorValency3}
In the differential of any graph with vertices of valency~$>2$, the blow up of a vertex of valency~$3$ %$\leqslant 3$ 
produces only the handshakes, that is the graphs which cancel out by Lemma~\ref{LemmaHandshakesCancel} (cf.\ footnote~\ref{FootWhyValencyThree} on p.~\pageref{FootWhyValencyThree} below).
\end{cor}

\begin{example}[$\Id\boldsymbol{\gamma}_5=0$]\label{ExDiffOf5Wheel}
The pentagon\/-\/wheel cocycle is the sum of two graphs with real coefficients
which is drawn in Fig.~\ref{FigPentaCocycle}.
%%%
\begin{figure}[htb]
\centerline{
$\boldsymbol{\gamma}_5={}$
\raisebox{0pt}[50pt][32pt]{
%%% The 5-wheel with labelling of vertices and edges. (p.8)
{\unitlength=0.5mm
\begin{picture}(58,53)(-33.5,-5)%(27.5,23)
\put(27.5,8.5){\circle*{2}}
\put(0,29.5){\circle*{2}}
\put(-27.5,8.5){\circle*{2}}
\put(-17.5,-23.75){\circle*{2}}
\put(17.5,-23.75){\circle*{2}}
\qbezier[70](27.5,8.5)(0,29.5)(0,29.5)
\qbezier[70](0,29.5)(-27.5,8.5)(-27.5,8.5)
\qbezier[70](-27.5,8.5)(-17.5,-23.75)(-17.5,-23.75)
\qbezier[70](-17.5,-23.75)(17.5,-23.75)(17.5,-23.75)
\qbezier[70](17.5,-23.75)(27.5,8.5)(27.5,8.5)
%%%
\put(0,0){\circle*{2}}
\qbezier[60](27.5,8.5)(0,0)(0,0)
\qbezier[60](0,29.5)(0,0)(0,0)
\qbezier[60](-27.5,8.5)(0,0)(0,0)
\qbezier[60](-17.5,-23.75)(0,0)(0,0)
\qbezier[60](17.5,-23.75)(0,0)(0,0)
%%% Vertices:
\put(30.5,8.7){{\tiny\textsf{1}}}
\put(2,31){{\tiny\textsf{2}}}
\put(-30.5,8.7){\llap{{\tiny\textsf{3}}}}
\put(-20.5,-27){\llap{{\tiny\textsf{4}}}}
\put(20.5,-27){{\tiny\textsf{5}}}
\put(2,4){{\tiny\textsf{6}}}
%%% Edges:
\put(13,21.5){{\tiny I}}
\put(-13,21.5){\llap{{\tiny II}}}
\put(-25.5,-10){\llap{{\tiny III}}}
\put(-3.5,-31){{\tiny IV}}
\put(25.5,-10){{\tiny V}}
\put(10,-3){{\tiny VI}}
\put(1.5,13.5){{\tiny VII}}
\put(-7,5.5){\llap{{\tiny VIII}}}
\put(-8.5,-10){\llap{{\tiny IX}}}
\put(7,-15){\llap{{\tiny X}}}
%%%
\end{picture}
}
}
$\phantom{m}{}+\dfrac{5}{2}\cdot{}$
\raisebox{0pt}[15pt][10pt]{ 
%%% The roof with labelling of vertices and edges. (p.8)
{\unitlength=0.72mm
%\raisebox{0pt}[1pt][1pt](
\begin{picture}(50,30)(-31,-4)%(27.5,23)
\put(12,0){\circle*{1.5}}
\put(-12,0){\circle*{1.5}}
\put(25,15){\circle*{1.5}}
\put(-25,15){\circle*{1.5}}
\put(-25,-15){\circle*{1.5}}
\put(25,-15){\circle*{1.5}}
%%%
\put(-12,0){\line(1,0){24}}
\put(-25,15){\line(1,0){50}}
\put(-25,-15){\line(1,0){50}}
\put(25,-15){\line(0,1){32}}
\put(-25,15){\line(0,-1){32}}
%%%
\qbezier[60](25,15)(12,0)(12,0)
\qbezier[60](-25,15)(-12,0)(-12,0)
\qbezier[60](-25,-15)(-12,0)(-12,0)
\qbezier[60](25,-15)(12,0)(12,0)
%%%
\put(12.5,17){\oval(25,10)[t]}
\put(-12.5,-17){\oval(25,10)[b]}
\put(0,2){\line(0,1){11}}
\put(0,-2){\line(0,-1){11}}
%%% Vertices:
\put(-27,-19){\llap{{\tiny\textsf{1}}}}
\put(27,-19){{\tiny\textsf{2}}}
\put(27,17){{\tiny\textsf{3}}}
\put(-27,17){\llap{{\tiny\textsf{4}}}}
\put(-11,2){{\tiny\textsf{5}}}
\put(11,2){\llap{{\tiny\textsf{6}}}}
%%% Edges:
\put(7,-19){{\tiny I}}
\put(27,-3){{\tiny II}}
\put(-13,17){{\tiny III}}
\put(-27,-3){\llap{{\tiny IV}}}
\put(-17,6){{\tiny V}}
\put(-17,-10){{\tiny VI}}
\put(-3,-4.5){\llap{{\tiny VII}}}
\put(18.5,8){\llap{{\tiny VIII}}}
\put(17.5,-10){\llap{{\tiny IX}}}
\put(-11,-21){\llap{{\tiny X}}}
\end{picture}
%}
}%
}
}
\caption{The Kontsevich\/--\/Willwacher pentagon\/-\/wheel cocycle~$\boldsymbol{\gamma}_5$.}\label{FigPentaCocycle}
\end{figure}
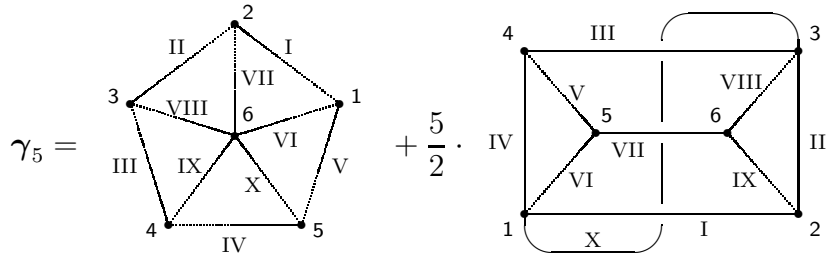
%%%
The edges in every %connected component
term are ordered by $I\wedge\cdots\wedge X$.
The differential of a sum of graphs is the sum of their differentials; this is why we calculate them separately and then collect similar terms. By the above, neither contains any leaves; likewise by the handshake Lemma~\ref{LemmaHandshakesCancel}, all the graphs --\,in which a new vertex (of valency~$2$) appears as midpoint of the already existing edge\,-- cancel.
By Corollary~\ref{CorValency3} it remains for us to consider the blow-ups of only the vertices of valency~$\geqslant 4$ (cf.~\cite{Ascona96}). Such are the axis vertex of the pentagon wheel and vertices labelled~\textsf{1} and~\textsf{3} in the other graph (the prism). By blowing up the pentagon wheel axis we shall obtain the (nonzero) `human' and the (zero) `monkey' graphs, presented %shown protrayed displayed
 in what follows. Likewise from the prism graph in~$\boldsymbol{\gamma}_5$ one obtains the `human', the `monkey', and the (zero) `stone'. Let us now discuss this in full detail.

From the pentagon wheel we obtain $2\cdot 5$ Da Vinci's `human' graphs, two of which are portrayed in Fig.~\ref{FigHumanAB}. %here. 
(The factor~$2$ \label{factor2} occurs from the two distinct ways to attach three versus two old edges in the wheel to the loose ends of the inserted edge $\bullet\!{-}\!\bullet$.) 
%%%
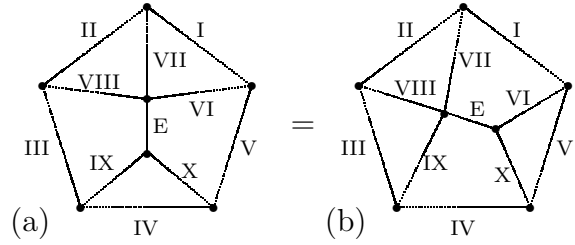
\begin{figure}[htb]
\centerline{
\raisebox{-1.5pt}[45pt][35pt]{
%%% The Human, standing right, with edges. (p.9)
{\unitlength=0.5mm
\begin{picture}(58,53)(-33.5,-5)%(27.5,23)
\put(27.5,8.5){\circle*{2}}
\put(0,29.5){\circle*{2}}
\put(-27.5,8.5){\circle*{2}}
\put(-17.5,-23.75){\circle*{2}}
\put(17.5,-23.75){\circle*{2}}
\qbezier[80](27.5,8.5)(0,29.5)(0,29.5)
\qbezier[80](0,29.5)(-27.5,8.5)(-27.5,8.5)
\qbezier[80](-27.5,8.5)(-17.5,-23.75)(-17.5,-23.75)
\qbezier[80](-17.5,-23.75)(17.5,-23.75)(17.5,-23.75)
\qbezier[80](17.5,-23.75)(27.5,8.5)(27.5,8.5)
%%%
\put(0,5){\circle*{2}}
\put(0,-9.5){\circle*{2}}
\put(0,5){\line(0,-1){14}}
\qbezier[60](27.5,8.5)(0,5)(0,5)
\qbezier[60](-27.5,8.5)(0,5)(0,5)
\qbezier[60](0,29.5)(0,5)(0,5)
\qbezier[60](-17.5,-23.75)(0,-9)(0,-9)
\qbezier[60](17.5,-23.75)(0,-9)(0,-9)
%%% Edges:
\put(13,21.5){{\tiny I}}
\put(-13,21.5){\llap{{\tiny II}}}
\put(-25.5,-10){\llap{{\tiny III}}}
\put(-3.5,-31){{\tiny IV}}
\put(24.5,-10){{\tiny V}}
\put(11,0.5){{\tiny VI}}
\put(1.5,13.5){{\tiny VII}}
\put(-7,7.5){\llap{{\tiny VIII}}}
\put(-8.5,-14){\llap{{\tiny IX}}}
\put(13.5,-16){\llap{{\tiny X}}}
\put(1.5,-4){{\tiny E}}
%%%
\put(-25.5,-30){\llap{(a)}}
\end{picture}
}
}
\phantom{l}
$=$
\raisebox{-1.5pt}[45pt][35pt]{
%%% The Human, turned by +72 degrees, with edges. (p.9)
{\unitlength=0.5mm
\begin{picture}(58,53)(-33.5,-5)%(27.5,23)
\put(27.5,8.5){\circle*{2}}
\put(0,29.5){\circle*{2}}
\put(-27.5,8.5){\circle*{2}}
\put(-17.5,-23.75){\circle*{2}}
\put(17.5,-23.75){\circle*{2}}
\qbezier[80](27.5,8.5)(0,29.5)(0,29.5)
\qbezier[80](0,29.5)(-27.5,8.5)(-27.5,8.5)
\qbezier[80](-27.5,8.5)(-17.5,-23.75)(-17.5,-23.75)
\qbezier[80](-17.5,-23.75)(17.5,-23.75)(17.5,-23.75)
\qbezier[80](17.5,-23.75)(27.5,8.5)(27.5,8.5)
%%%
\put(-5,1){\circle*{2}}
\put(8.5,-3){\circle*{2}}
\qbezier[60](-5,1.5)(8.5,-3)(8.5,-3)
%%%
\qbezier[60](27.5,8.5)(8.5,-3)(8.5,-3)
\qbezier[60](-27.5,8.5)(-5,1.5)(-5,1.5)
\qbezier[60](0,29.5)(-5,1.5)(-5,1.5)
\qbezier[60](-17.5,-23.75)(-5,1.5)(-5,1.5)
\qbezier[60](17.5,-23.75)(8.5,-3)(8.5,-3)
%%% Edges:
\put(13,21.5){{\tiny I}}
\put(-13,21.5){\llap{{\tiny II}}}
\put(-25.5,-10){\llap{{\tiny III}}}
\put(-3.5,-31){{\tiny IV}}
\put(24.5,-10){{\tiny V}}
\put(11,3.5){{\tiny VI}}
\put(-1.5,13.5){{\tiny VII}}
\put(-7,6){\llap{{\tiny VIII}}}
\put(-4,-14){\llap{{\tiny IX}}}
\put(12.5,-17){\llap{{\tiny X}}}
\put(1.5,0.5){{\tiny E}}
%%%
\put(-25.5,-30){\llap{(b)}}
\end{picture}
}
}
}
\caption{Two of the fourteen Da Vinci's `human' graphs occurring with weights in~$\Id\boldsymbol{\gamma}_5$.}\label{FigHumanAB}
\end{figure}
%%%
We claim that all the five `human' graphs (i.e.\ standing with their feet on the edges~$I$, $\ldots$, $V$ in the pentagon wheel) carry the same sign, providing the overall coefficient $+10=2\cdot(+5)$ %not (-5)
of such graph in the differential of the wheel. The graph~(b) 
   %\marginpar{create (b) and (a) and refer} 
is topologically equal to the graph~(a); indeed, the matching of their edges is $I^{(b)} = V^{(a)}$, $II^{(b)} = I^{(a)}$, $III^{(b)} = II^{(a)}$, $IV^{(b)} = III^{(a)}$, $V^{(b)} = IV^{(a)}$, $VI^{(b)} = X^{(a)}$, $VII^{(b)} = VI^{(a)}$, $VIII^{(b)} = VII^{(a)}$, $IX^{(b)} = VIII^{(a)}$, and $X^{(b)} = IX^{(a)}$; also $E^{(b)} = E^{(a)}$. Hence the postulated ordering of edges in~(b) is 
%\begin{align*}
\begin{multline}\label{EqMaching}
E^{(b)}\wedge I^{(b)}\wedge \cdots\wedge X^{(b)} =  
E^{(a)}\wedge V^{(a)}\wedge I^{(a)}\wedge II^{(a)}\wedge III^{(a)}\wedge IV^{(a)}\wedge
\\
\wedge X^{(a)}\wedge VI^{(a)}\wedge VII^{(a)}\wedge VIII^{(a)}\wedge IX^{(a)}
= + E^{(a)}\wedge I^{(a)}\wedge \cdots\wedge X^{(a)},
\end{multline}
%\end{align*} 
which equals the edge ordering of the graph~(a). For the other three graphs of this shape the equalities of wedge products are similar: a parity\/-\/even permutation of edges works out the mapping of graphs, e.g., to the graph~(a) which we take as the reference. %standard

%page 10 ...
From the pentagon wheel we also obtain $2\cdot 5$ `monkey' graphs, a specimen of which is shown 
in Fig.~\ref{FigMonkey} below. %picture
   %\marginpar{To do picture monkey and redrawing}
\begin{figure}[htb]
\centerline{
\raisebox{-1.5pt}[45pt][35pt]{
%%% The Monkey, standing right, with edges. (p.10)
{\unitlength=0.5mm
\begin{picture}(58,53)(-31.5,-5)%(27.5,23)
\put(27.5,8.5){\circle*{2}}
\put(0,29.5){\circle*{2}}
\put(-27.5,8.5){\circle*{2}}
\put(-17.5,-23.75){\circle*{2}}
\put(17.5,-23.75){\circle*{2}}
\qbezier[80](27.5,8.5)(0,29.5)(0,29.5)
\qbezier[80](0,29.5)(-27.5,8.5)(-27.5,8.5)
\qbezier[80](-27.5,8.5)(-17.5,-23.75)(-17.5,-23.75)
\qbezier[80](-17.5,-23.75)(17.5,-23.75)(17.5,-23.75)
\qbezier[80](17.5,-23.75)(27.5,8.5)(27.5,8.5)
%%%
\put(0,5){\circle*{2}}
\put(0,-9.5){\circle*{2}}
\put(0,5){\line(0,-1){14}}
\qbezier[60](27.5,8.5)(0,5)(0,5)
\qbezier[60](-27.5,8.5)(0,-9)(0,-9)
\qbezier[60](0,29.5)(0,5)(0,5)
\qbezier[60](-17.5,-23.75)(0,5)(0,5)
\qbezier[60](17.5,-23.75)(0,-9)(0,-9)
%%% Edges:
\put(13,21.5){{\tiny I}}
\put(-13,21.5){\llap{{\tiny II}}}
\put(-25.5,-10){\llap{{\tiny III}}}
\put(-3.5,-31){{\tiny IV}}
\put(24.5,-10){{\tiny V}}
\put(11,0.5){{\tiny VI}}
\put(1.5,13.5){{\tiny VII}}
\put(-8,3.5){\llap{{\tiny VIII}}}
\put(-4.5,-18){\llap{{\tiny IX}}}
\put(13.5,-17){\llap{{\tiny X}}}
\put(1.5,-4){{\tiny E}}
%%%
\end{picture}
}
}
\phantom{ll}$=$
\raisebox{-1.5pt}[45pt][35pt]{
%%% The Monkey in fully symmetric shape (p.10)
{\unitlength=0.6mm
\begin{picture}(58,53)(-31.5,-13)%(27.5,23)
\put(0,0){\circle*{2}}
\put(15,10){\circle*{2}}
\put(-15,10){\circle*{2}}
\put(25,-10){\circle*{2}}
\put(-25,-10){\circle*{2}}
\put(10,-30){\circle*{2}}
\put(-10,-30){\circle*{2}}
%%%
\put(-15,10){\line(1,0){30}}
\put(-15,10){\line(3,-2){15}}
\put(15,10){\line(-3,-2){15}}
\put(-15,10){\line(-1,-2){10}}
\put(15,10){\line(1,-2){10}}
\put(-25,-10){\line(3,-4){15}}
\put(25,-10){\line(-3,-4){15}}
\qbezier[90](-25,-10)(10,-30)(10,-30)
\qbezier[90](25,-10)(-10,-30)(-10,-30)
\put(0,0){\line(-1,-3){10}}
\put(0,0){\line(1,-3){10}}
\put(-1,10.5){{\tiny I}}
\put(-21,-2){\llap{{\tiny II}}}
\put(-17,-26){\llap{{\tiny III}}}
\put(17,-26){{\tiny X}}
\put(21,-2){{\tiny V}}
\put(6,0){{\tiny VI}}
\put(-6,0){\llap{{\tiny VII}}}
\put(-16,-15){{\tiny VIII}}
\put(-4,-8){\llap{{\tiny IX}}}
\put(16,-15){\llap{{\tiny IV}}}
\put(4,-8){{\tiny E}}
\end{picture}
}
}
\phantom{l}$=0$}
\caption{The `monkey' graph: animal touches earth with its palm; this is an example of zero graph.}\label{FigMonkey}
\end{figure}
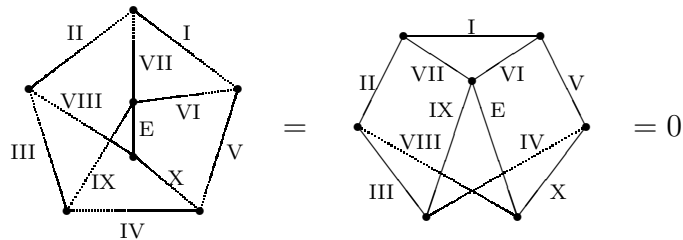   
Note that the `monkey' graph is mirror\/-\/symmetric, see the redrawing.
This symmetry induces a permutation of edges which swaps~5 pairs, so (since 5~is odd) the `monkey' graph is equal to zero. %page 11 ???

%page 11
Now consider the graphs obtained by blowing up vertices~\textsf{1} and~\textsf{3} in the prism graph. How are the four old neighbors distributed over the ends of the inserted edge? Whenever those four old neighbours are distributed in proportion $4=3+1$ (i.e.\ with valencies~$4$ and~$2$ for the two ends of the inserted edge), there is no contribution from the resulting graphs to~$\Id(\text{prism})$ by the handshake Lemma~\ref{LemmaHandshakesCancel}. So the graphs which could contibute are only those with the $4=2+2$ distribution (i.e.\ with valency~$3$ for either of the ends of the inserted edge). For one fixed neighbour of one of the new edge's ends there are three ways to choose the second neighbour of that vertex. This is how the `human', `monkey', and `stone' graphs are presently obtained.

%page 12 = 11
Let us blow up vertex~\textsf{1} in the prism in these three different ways. First we make the end (now marked~\textsf{1}) of the inserted edge adjacent to~\textsf{2} and~\textsf{3}, and the other end (marked~\textsf{1}$'$)
to vertices~\textsf{4} and~\textsf{5}; the resulting graph is the `human' graph shown in Fig.~\ref{FigHumanFromRoof}.
%%%
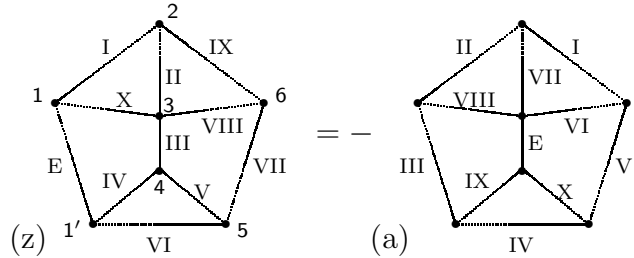
\begin{figure}[htb]
\centerline{
\raisebox{-4.5pt}[47pt][35pt]{
%%% The Human-from-Roof, standing right, with [roof] vertices and [roof] edges. (p.11)
{\unitlength=0.5mm
\begin{picture}(58,53)(-27.5,-5)%(27.5,23)
\put(27.5,8.5){\circle*{2}}
\put(0,29.5){\circle*{2}}
\put(-27.5,8.5){\circle*{2}}
\put(-17.5,-23.75){\circle*{2}}
\put(17.5,-23.75){\circle*{2}}
\qbezier[80](27.5,8.5)(0,29.5)(0,29.5)
\qbezier[80](0,29.5)(-27.5,8.5)(-27.5,8.5)
\qbezier[80](-27.5,8.5)(-17.5,-23.75)(-17.5,-23.75)
\qbezier[80](-17.5,-23.75)(17.5,-23.75)(17.5,-23.75)
\qbezier[80](17.5,-23.75)(27.5,8.5)(27.5,8.5)
%%%
\put(0,5){\circle*{2}}
\put(0,-9.5){\circle*{2}}
\put(0,5){\line(0,-1){14}}
\qbezier[60](27.5,8.5)(0,5)(0,5)
\qbezier[60](-27.5,8.5)(0,5)(0,5)
\qbezier[60](0,29.5)(0,5)(0,5)
\qbezier[60](-17.5,-23.75)(0,-9)(0,-9)
\qbezier[60](17.5,-23.75)(0,-9)(0,-9)
%%% Edges:
\put(13,21.5){{\tiny IX}}
\put(-13,21.5){\llap{{\tiny I}}}
\put(-25.5,-10){\llap{{\tiny E}}}
\put(-3.5,-31){{\tiny VI}}
\put(24.5,-10){{\tiny VII}}
\put(11,0.5){{\tiny VIII}}
\put(1.5,13.5){{\tiny II}}
\put(-7,7.5){\llap{{\tiny X}}}
\put(-8.5,-14){\llap{{\tiny IV}}}
\put(13.5,-17){\llap{{\tiny V}}}
\put(1.5,-4){{\tiny{III}}}
%%% Vertices:
\put(30.5,8.7){{\tiny\textsf{6}}}
\put(2,31){{\tiny\textsf{2}}}
\put(-30.5,8.7){\llap{{\tiny\textsf{1}}}}
\put(-20.5,-27){\llap{{\tiny\textsf{1}$'$}}}
\put(20.5,-27){{\tiny\textsf{5}}}
\put(-1.5,-15.5){{\tiny\textsf{4}}}
\put(1,6){{\tiny\textsf{3}}}
%%%
\put(-29.5,-30){\llap{(z)}}
\end{picture}
}
}
%%%
${}=-{}$
%%%
\raisebox{-4.5pt}[47pt][35pt]{
%%% The Human-from-Roof, standing right, without [roof] vertices and with [pentagon] edges. (p.11)
{\unitlength=0.5mm
\begin{picture}(58,53)(-33.5,-5)%(27.5,23)
\put(27.5,8.5){\circle*{2}}
\put(0,29.5){\circle*{2}}
\put(-27.5,8.5){\circle*{2}}
\put(-17.5,-23.75){\circle*{2}}
\put(17.5,-23.75){\circle*{2}}
\qbezier[80](27.5,8.5)(0,29.5)(0,29.5)
\qbezier[80](0,29.5)(-27.5,8.5)(-27.5,8.5)
\qbezier[80](-27.5,8.5)(-17.5,-23.75)(-17.5,-23.75)
\qbezier[80](-17.5,-23.75)(17.5,-23.75)(17.5,-23.75)
\qbezier[80](17.5,-23.75)(27.5,8.5)(27.5,8.5)
%%%
\put(0,5){\circle*{2}}
\put(0,-9.5){\circle*{2}}
\put(0,5){\line(0,-1){14}}
\qbezier[60](27.5,8.5)(0,5)(0,5)
\qbezier[60](-27.5,8.5)(0,5)(0,5)
\qbezier[60](0,29.5)(0,5)(0,5)
\qbezier[60](-17.5,-23.75)(0,-9)(0,-9)
\qbezier[60](17.5,-23.75)(0,-9)(0,-9)
%%% Edges:
\put(13,21.5){{\tiny I}}
\put(-13,21.5){\llap{{\tiny II}}}
\put(-25.5,-10){\llap{{\tiny III}}}
\put(-3.5,-31){{\tiny IV}}
\put(24.5,-10){{\tiny V}}
\put(11,0.5){{\tiny VI}}
\put(1.5,13.5){{\tiny VII}}
\put(-7,7.5){\llap{{\tiny VIII}}}
\put(-8.5,-14){\llap{{\tiny IX}}}
\put(13.5,-17){\llap{{\tiny X}}}
\put(1.5,-4){{\tiny{E}}}
%%% Vertices:
%\put(30.5,8.7){{\tiny\textsf{6}}}
%\put(2,31){{\tiny\textsf{2}}}
%\put(-30.5,8.7){\llap{{\tiny\textsf{1}}}}
%\put(-20.5,-27){\llap{{\tiny\textsf{1}$'$}}}
%\put(20.5,-27){{\tiny\textsf{5}}}
%\put(-1.5,-15.5){{\tiny\textsf{4}}}
%\put(1,6){{\tiny\textsf{3}}}
%%%
\put(-29.5,-30){\llap{(a)}}
\end{picture}
}
}
}
\caption{One of the `human' graphs obtained by blowing up --\,according to a scenario 
discussed in the text\,-- a vertex of valency four in the prism graph from~$\boldsymbol{\gamma}_5$.}\label{FigHumanFromRoof}
\end{figure}
%%%
%page 10 cont.
From the prism graph we obtain $2\cdot 2=4$ such `human' graphs. One of the factors~$2$ is obtained like before, namely by attaching a given set of old edges to one or the other end of the inserted edge $\bullet\!\!\!-\!\!\!\bullet$, 
   see p.~\pageref{factor2}; %just a bit above ???
the other factor~$2$ comes by the rotational symmetry of the prism graph.
%The `human' graph is obtained from blowing up vertex 1 (and vertex 3 likewise) %by 
%shown in \dots \marginpar{To do picture human from roof}.
%page 12 = 12
Indeed, the prism with one diagonal is symmetric under the rotation by angle~$\pi$ that %brings
transposes the vertices $\mathsf{1} \rightleftarrows \mathsf{3}$, $\mathsf{2} \rightleftarrows \mathsf{4}$, and $\mathsf{5}\rightleftarrows \mathsf{6}$. This is why the same `human' graph is obtained when the vertex~\textsf{3} is blown up according to a similar scenario. We claim that the permutation of edges that relates the two graphs is 
parity\/-\/even (similar to~\eqref{EqMaching}), so they do not cancel but add~up.
Summarizing, the overal coefficient of the `human' graph --\,produced in~$\Id(\text{prism})$ for the edge ordering $E\wedge I\wedge\cdots \wedge X$ shown in Fig.~\ref{FigHumanFromRoof}\,-- equals~$2\cdot 2=+4$.
%\marginpar{To do picture human from prism} 

The count of an overall contribution $10 + \frac{5}{2}\cdot (+4)\cdot (-1 \text{ from edge ordering}) =0$ to the differential $\Id(\boldsymbol{\gamma}_5)$ of the cocycle~$\boldsymbol{\gamma}_5$ will be performed %later; 
using Eq.~\eqref{EqPentaMeetsRoof};
right now let us inspect the vanishing of contributions from the other two types of graphs wich are obtained by the two possible edge distribution scenarios (with respect to the ends of the new edge $\bullet\!\!{-}\!\!\bullet$ that replaces the blown-up vertex~\textsf{1} or~\textsf{3} in the prism).

\smallskip
%page 13
\noindent%
%%% The (other) Monkey, standing right, with newly labelled edges. (p.13)
{\unitlength=0.5mm
\begin{picture}(0,0)(-278,-5)%(55,53)(-25,-25)%(27.5,23)
\put(27.5,8.5){\circle*{2}}
\put(0,29.5){\circle*{2}}
\put(-27.5,8.5){\circle*{2}}
\put(-17.5,-23.75){\circle*{2}}
\put(17.5,-23.75){\circle*{2}}
\qbezier[80](27.5,8.5)(0,29.5)(0,29.5)
\qbezier[80](0,29.5)(-27.5,8.5)(-27.5,8.5)
\qbezier[80](-27.5,8.5)(-17.5,-23.75)(-17.5,-23.75)
\qbezier[80](-17.5,-23.75)(17.5,-23.75)(17.5,-23.75)
\qbezier[80](17.5,-23.75)(27.5,8.5)(27.5,8.5)
%%%
\put(0,5){\circle*{2}}
\put(0,-9.5){\circle*{2}}
\put(0,5){\line(0,-1){14}}
\qbezier[60](27.5,8.5)(0,5)(0,5)
\qbezier[60](-27.5,8.5)(0,-9)(0,-9)
\qbezier[60](0,29.5)(0,5)(0,5)
\qbezier[60](-17.5,-23.75)(0,5)(0,5)
\qbezier[60](17.5,-23.75)(0,-9)(0,-9)
%%% Edges:
\put(13,21.5){{\tiny IX}}
\put(-13,21.5){\llap{{\tiny I}}}
\put(-25.5,-10){\llap{{\tiny E}}}
\put(-3.5,-31){{\tiny VI}}
\put(24.5,-10){{\tiny VII}}
\put(11,0.5){{\tiny VIII}}
\put(1.5,13.5){{\tiny II}}
\put(-9,2.5){\llap{{\tiny IV}}}
\put(-5.5,-18){\llap{{\tiny X}}}
\put(13.5,-17){\llap{{\tiny V}}}
\put(1.5,-4){{\tiny III}}
%%% Vertices:
\put(30.5,8.7){{\tiny\textsf{6}}}
\put(2,31){{\tiny\textsf{2}}}
\put(-30.5,8.7){\llap{{\tiny\textsf{1}}}}
\put(-20.5,-27){\llap{{\tiny\textsf{1}$'$}}}
\put(20.5,-27){{\tiny\textsf{5}}}
\put(-1.5,-15.5){{\tiny\textsf{4}}}
\put(1,6){{\tiny\textsf{3}}}
\end{picture}%
}%
\parbox{115mm}{{}\quad
The `monkey' graph is obtained by blowing up the vertex~\textsf{1} (or~\textsf{3}) in the prism and then attaching the new edge's end, still marked~\textsf{1}, to the vertices~\textsf{2} and~\textsf{4}. The other end, now marked~\textsf{1}$'$, of the new edge becomes adjacent to the vertices~\textsf{3} and~\textsf{5}. We keep in mind that every `monkey' graph itself is equal to zero, hence no contribution to~$\Id(\text{prism})$ occurs.}

\smallskip
So far, the new vertex~\textsf{1} has always been a fixed neighbour of vertex~\textsf{2}, and it was made adjacent to~\textsf{3} in the `human' and to~\textsf{4} in the `monkey' graphs, respectively. The overall set of neigbours of the new edge \textsf{1}--\textsf{1}$'$, apart from the fixed vertex~\textsf{2}, consists of vertices~\textsf{3}, \textsf{4} and~\textsf{5}. So the third scenario to consider is the `stone' graph 
in which the new vertex~\textsf{1} is adjacent to~\textsf{1}$'$, \textsf{2}, and~\textsf{5}, whereas the new vertex~\textsf{1}$'$ neighbours~\textsf{1}, \textsf{3}, and~\textsf{4}.
   %\marginpar{To do picture "..." + choose in text: butterfly or diamond} 
\begin{figure}[htb]%"Not Monkey"
\centerline{ 
%%% The NOT-Monkey = Butterfly = Diamond, with edges (p.13)
\raisebox{2pt}[44pt][24pt]{
{\unitlength=0.8mm
\begin{picture}(40,35)(-20,0)
\put(10,0){\circle*{1.2}}
\put(-10,0){\circle*{1.2}}
\put(0,10){\circle*{1.2}}
\put(20,0){\circle*{1.2}}
\put(-20,0){\circle*{1.2}}
\put(0,20){\circle*{1.2}}
\put(0,-15){\circle*{1.2}}
%%%
\put(0,20){\line(1,-1){20 } }
\put(0,20){\line(-1,-1){20 } }
\put(0,20){\line(0,-1){10 } }
\put(0,10){\line(1,-1){10 } }
\put(0,10){\line(-1,-1){10 } }
\put(10,0){\line(1,0){10 } }
\put(-10,0){\line(-1,0){10 } }
\put(-20,0){\line(4,-3){20 } }
\put(20,0){\line(-4,-3){20 } }
\put(-10,0){\line(2,-3){10 } }
\put(10,0){\line(-2,-3){10 } }
%%%
\put(10,-11){{\tiny I}}
\put(10,11){{\tiny II}}
\put(-10,11){\llap{\tiny III}}
\put(-10,-11){\llap{\tiny IV}}
\put(5,-7){\llap{\tiny V}}
\put(11.5,0.5){{\tiny VI}}
\put(0.5,11.5){{\tiny VII}}
\put(-10.75,0.5){\llap{\tiny VIII}}
\put(-5,-7){{\tiny IX}}
\put(-5.5,5){\llap{\tiny X}}
\put(5.5,5){{\tiny E}}
% Vertices:
\put(1,20.5){{\tiny\textsf{5}}}
\put(21,0.5){{\tiny\textsf{4}}}
\put(-21,0.5){\llap{{\tiny\textsf{6}}}}
\put(-1,6){{\tiny\textsf{1}}}
\put(-8,-1.5){{\tiny\textsf{2}}}
\put(8,-1.5){\llap{{\tiny\textsf{1}$'$}}}
\put(1,-16.5){{\tiny\textsf{3}}}
\end{picture}
}
}
$=0$
}
\end{figure}
%%%
This graph is mirror\/-\/symmetric under the transposition of vertices $\mathsf{1}'\rightleftarrows \mathsf{2}$ and $\mathsf{4} \rightleftarrows \mathsf{6}$, which induces the swaps in five edge pairs, namely, $II \rightleftarrows III$, $E\rightleftarrows X$, $VI\rightleftarrows VIII$, $V\rightleftarrows IX$, and $I\rightleftarrows IV$. Arguing as before, we deduce that every such `stone' graph (obtained by a blow up of either~\textsf{1} or~\textsf{3} in the prism) is~zero.

Our final task in the calculation of~$\Id(\boldsymbol{\gamma}_5)$ is collecting the coefficients of the `human' graphs from~$\Id(\text{5-\/wheel})$ and~$\Id(\text{prism})$, coming not only with coefficients~$10$ and~$4$ respectively, but also with the respective edge orderings. To discriminate edges between the two pictures, that is originating from the pentagon wheel and the prism, let us use the superscripts~($a$) and~($z$), %respectively
see Fig.~\ref{FigHumanFromRoof}.
The edge matching is $E^{(z)} = III^{(a)}$, $I^{(z)} = II^{(a)}$, $II^{(z)} = VII^{(a)}$, $III^{(z)} = E^{(a)}$, $IV^{(z)} = IX^{(a)}$, $V^{(z)} = X^{(a)}$, $VI^{(z)} = IV^{(a)}$, $VII^{(z)} = V^{(a)}$, $VIII^{(z)} = VI^{(a)}$, $IX^{(z)} = I^{(a)}$, and $X^{(z)} = VIII^{(a)}$. Consequently, for the edge orderings we have
\begin{multline}\label{EqPentaMeetsRoof}
E^{(z)}\wedge I^{(z)}\wedge\cdots\wedge X^{(z)} ={} \\
III^{(a)}\wedge II^{(a)}\wedge VII^{(a)}\wedge E^{(a)}\wedge IX^{(a)}\wedge X^{(a)}\wedge IV^{(a)}\wedge V^{(a)}\wedge VI^{(a)}\wedge I^{(a)}\wedge VIII^{(a)}\\
{} = (-)^{23} \, E^{(a)}\wedge I^{(a)}\wedge \cdots\wedge X^{(a)}. 
\end{multline}
This argument shows that the graph differential of the linear combination $(+1)\cdot{}$pentagon\/-\/wheel $+$ $\frac{5}{2}\cdot{}$ prism,
with either graph's edge ordering specified as in Example~\ref{Ex5Wheel},
%\ref{EqMaching} 
vanishes. In other words, $\boldsymbol{\gamma}_5$ is a $\Id$-\/cocycle.
\end{example}

%Main result
\section{A representative of the heptagon\/-\/wheel cocycle~$\boldsymbol{\gamma}_7$}\label{SecResult}
\noindent%
It is already known that the heptagon\/-\/wheel cocycle~$\boldsymbol{\gamma}_7$, the existence of which was stated %predicted 
in Theorem~\textup{\ref{ThWillwacherWheels}}, is unique modulo $\Id$-\/trivial terms in the respective cohomology group of connected graphs on $8$~vertices and $14$~edges \textup{(}hence with $7$~basic \mbox{loops),~cf.\:\cite{KhoroshkinWillwacherZivkovic}.}

\begin{theor}\label{ThHepta}
The encoding of every term in a %our 
representative of the cocycle~$\boldsymbol{\gamma}_7$ is given in Table~\textup{\ref{TabHeptagonCocycle}}, 
the format of lines in which is the lexicographic\/-\/ordered list of fourteen edges $I \wedge \cdots \wedge XIV$ followed by the nonzero real coefficient.
The forty\/-\/six graphs that form this representative of the $\Id$-\/cohomology class~$\boldsymbol{\gamma}_7$  are shown on pages~\textup{\pageref{pStartHepta}--\pageref{pEndHepta}.}
\end{theor}

\begin{table}[htb]
\caption{The heptagon\/-\/wheel graph
 cocycle~$\boldsymbol{\gamma}_7$.}\label{TabHeptagonCocycle}
{\upshape\tiny%\footnotesize% %\fontsize{6}{1}\selectfont 
\begin{tabular}{l r c l r} 
\hline
Graph encoding\strut & Coeff. & & Graph encoding & Coeff.\\[0.5pt]
\hline
16	17	18	23	25	28	34	38	46	48	57	58	68	78	& 	$\mathbf{1}\phantom{/1}$		& & 	12	13	18	25	26	37	38	45	46	47	56	57	68	78	&	$-7\phantom{/1}$	\\
12	14	18	23	27	35	37	46	48	57	58	67	68	78	& 	$-21/8$ &	& 	12	14	16	23	25	36	37	45	48	57	58	67	68	78	&	$77/8$	\\
13	14	18	23	25	28	37	46	48	56	57	67	68	78	& 	$-77/4$ &	& 	13	16	17	24	25	26	35	37	45	48	58	67	68	78	&	$-7\phantom{/1}$	\\
12	13	15	24	27	35	36	46	48	57	58	67	68	78	& 	$-35/8$ &	& 	14	15	17	23	26	28	37	38	46	48	56	57	68	78	&	$49/4$	\\
12	13	18	24	26	37	38	46	47	56	57	58	68	78	& 	$49/8$ &	& 	12	16	18	27	28	34	36	38	46	47	56	57	58	78	&	$-147/8$	\\
14	17	18	23	25	26	35	37	46	48	56	58	67	78	& 	$77/8$ &	& 	12	15	16	27	28	35	36	38	45	46	47	57	68	78	&	$-21/8$	\\
12	13	18	26	27	35	38	45	46	47	56	57	68	78	& 	$-105/8$	 &	& 	12	14	18	23	27	35	36	45	46	57	58	67	68	78	&	$-35/8$	\\
12	14	18	23	27	36	38	46	48	56	57	58	67	78	& 	$7/8$ &	& 	14	15	16	23	26	28	37	38	46	48	57	58	67	78	&	$-49/4$	\\
12	14	15	23	27	35	36	46	48	57	58	67	68	78	& 	$35/8$ &	& 	12	15	18	23	28	34	37	46	48	56	57	67	68	78	&	$105/8$	\\
12	13	14	27	28	36	38	46	47	56	57	58	68	78	& 	$-49/8$ &	& 	12	14	17	23	26	37	38	46	48	56	57	58	68	78	&	$-49/8$	\\
12	13	18	25	27	34	36	47	48	56	58	67	68	78	& 	$35/4$ &	& 	12	16	18	25	27	35	36	37	45	46	48	57	68	78	&	$49/1\lefteqn{6}$	\\
12	13	14	25	26	36	38	45	47	57	58	67	68	78	& 	$-119/1\lefteqn{6}$ &	& 	12	13	18	25	27	35	36	46	47	48	56	57	68	78	&	$7\phantom{/1}$	\\
12	13	15	24	28	36	38	47	48	56	57	67	68	78	& 	$49/8$ &	& 	12	14	18	25	28	34	36	38	47	57	58	67	68	78	&	$-7\phantom{/1}$	\\
12	13	14	23	28	37	46	48	56	57	58	67	68	78	& 	$77/4$ &	& 	12	16	18	25	27	35	36	37	45	46	48	58	67	78	&	$-77/1\lefteqn{6}$	\\
12	15	17	25	26	35	36	38	45	47	48	67	68	78	& 	$-49/8$ &	& 	12	14	18	23	27	35	38	46	47	57	58	67	68	78	&	$77/4$	\\
13	15	18	24	26	28	37	38	46	47	56	57	68	78	& 	$-49/4$ &	& 	12	14	15	23	27	36	38	46	48	57	58	67	68	78	&	$35/2$	\\
13	14	18	25	26	28	36	38	47	48	56	57	67	78	& 	$-49/4$ &	& 	12	13	18	25	27	34	36	46	48	57	58	67	68	78	&	$-105/8$	\\
12	14	18	23	28	35	37	46	48	56	57	67	68	78	& 	$-7\phantom{/1}$ & &	12	15	16	25	27	35	36	38	46	47	48	57	68	78	&	$-7\phantom{/1}$	\\
12	14	18	23	28	36	38	46	47	56	57	58	67	78	& 	$-7\phantom{/1}$ &		& 	12	13	16	25	28	34	37	47	48	57	58	67	68	78	&	$-147/1\lefteqn{6}$	\\
12	15	16	25	27	35	36	38	46	47	48	58	67	78	& 	$49/8$ &	& 	12	13	17	25	26	35	37	45	46	48	58	67	68	78	&	$-77/4$	\\
12	14	18	23	28	36	37	46	47	56	57	58	68	78	& 	$49/8$ &		& 	12	14	17	23	27	35	38	46	48	57	58	67	68	78	&	$-49/8$	\\
12	13	15	26	27	35	36	45	47	48	58	67	68	78	& 	$-7\phantom{/1}$ &	& 	12	13	15	26	28	35	37	45	46	47	58	67	68	78	&	$-7/4$	\\
12	13	18	24	28	35	38	46	47	57	58	67	68	78	& 	$7\phantom{/1}$ &	& 	12	14	18	23	26	36	38	47	48	56	57	58	67	78	&	$-7\phantom{/1}$ 	\\
\hline
\end{tabular}} 
\end{table}

\begin{proof}[Proof scheme] % scheme
This reasoning is computer\/-\/assisted.
First, all connected graphs on $8$~vertices and $14$~edges, and without multiple edges were generated.
(There are $1579$~such graphs; note that arbitrary valency $N(v) \geqslant 1$ of vertices was allowed.)
The coefficient of the heptagon wheel was set equal to~$+1$, all other coefficients still to be determined.
After calculating the differential of the sum of all these weighted graphs (we used a program in {\sc Sage}, see Appendix~\ref{AppSage}), zero graphs were eliminated and the remaining terms were collected (in the same way as is explained in~\S\ref{SecExamples}).
In the resulting sum of weighted graphs on $9$~vertices and $15$~edges, we equated each coefficient to zero.
We solved this linear algebraic system w.r.t.\ the coefficients of graphs in $\boldsymbol{\gamma}_7$.
There are $N_{\text{im}}(7) = 35$ free parameters in the general solution; such parameters count the coboundaries which cannot modify the cohomology class marked by any particular representative (see Table~\ref{TabCount} on p.~\pageref{TabCount} below).
Therefore the solution~$\boldsymbol{\gamma}_7$ is unique modulo $\Id$-exact terms.
All those free parameters are now set %[equal] 
to zero and the resulting nonzero values of the graph coefficients are listed in Table~\ref{TabHeptagonCocycle}.
\end{proof}

\begin{proposition}[{see~\cite[Table~1]{KhoroshkinWillwacherZivkovic}}]\label{ThCount}
The space of nontrivial $\Id$-\/cocycles which are built of connected graphs on $n$~vertices and $2n-2$~ edges at~$1 \leqslant n \leqslant 9$ is spanned by the terahedron~$\boldsymbol{\gamma}_3$, pentagon\/-\/wheel cocycle~$\boldsymbol{\gamma}_5$ that consists of two graphs \textup{(}see Example~\textup{\ref{Ex5Wheel})}, heptagon\/-\/wheel cocycle~$\boldsymbol{\gamma}_7$ from Theorem~\textup{\ref{ThHepta},}
and the Lie bracket~$[\boldsymbol{\gamma}_3,\boldsymbol{\gamma}_5]$.
At the same time, %In particular, 
for either~$n=5$ or $n=7$, the respective graph cohomology groups are trivial.\footnote{None of the %computational 
results in Theorem~\ref{ThHepta} and Proposition~\ref{ThCount} involves floating point operations in the way how it is obtained; hence even if computer\/-\/assisted, both the claims are exact.}
\end{proposition}

%This result completely agrees with a computation of the graph cohomology group dimensions in~\cite[Table~1]{KhoroshkinWillwacherZivkovic}; the two reasonings verify each other.

\begin{proof}[Verification] %[Proof scheme] % scheme
The dimension $N_{\text{ker}}$ of the space of cocycles built of connected graphs $\boldsymbol{\gamma}$ on $n$ vertices and $2n-2$ edges is equal to the number of free parameters in the general solution to the linear system $\Id(\text{sum of such graphs } \gamma \text{ with undetermined coefficients}) = 0$.
At the same time, to determine the dimension $N_{\text{im}}$ of the subspace of coboundaries $\gamma = \Id(\delta)$, i.e.\ of those cocycles which are the differentials of connected graphs on $n-1$ vertices and $2n-3$ edges, we first count the number of $N_\delta$ of nonzero connected graphs $\delta$ in that vertex-edge bi-grading.
Then we subtract from $N_\delta$ the number $N_0$ of free parameters in the general solution to the linear algebraic system $\Id(\text{sums of such graphs }\delta\text{ with undetermined coefficients}) = 0$.
This subtrahend counts the number of relations between exact terms $\gamma = \Id(\delta)$; for $n < 9$ it is zero.
The dimension of cohomology group $H^*(n)$ in bi-grading $(n, 2n-2)$ is then $N_{\text{ker}} - N_{\text{im}} = N_{\text{ker}} - (N_\delta - N_0)$.

Our present count of the overall number of connected graphs (and of the zero graphs among them) and the dimensions $N_{\text{ker}}, N_\delta, N_0$ and $N_{\text{im}}$ of the respective vector spaces are summarized in Tables~\ref{TabCount} and~\ref{TabCountNgt2}.
% [below/on p.~\pageref{TabCount} below].
\end{proof}

%%%%%%%%%%%%%%%%%%%%%%%%%%%%%%%%%%%%%%%%%%%%
\begin{table}[htb]
\caption{Dimensions of connected graph spaces and cohomology groups.}\label{TabCount}
\begin{tabular}{|c | c | c | c | c | c | c | c | c | c| c |}
\hline
\multicolumn{2}{|c|}{$n$} & $\#E$ & $\#(\text{graphs})$ & $\#(=0)$ & \multicolumn{2}{|c|}{$\#(\neq 0)$, $N_\delta$} & \multicolumn{2}{|c|}{$N_{\text{ker}}$, $N_0$} & $N_{\text{im}}$ & $\dim H^*(n)$ \\
\hline
4 & & 6 & 1 & 0 & 1 & & 1 & & & 1 \\
& 3 & 5 & 0 & -- & -- & -- & & -- & -- & \\
\hline
5 & & 8 & 2 & 2 & 0 & & -- & & & 0 \\
& 4 & 7 & 0 & -- & -- & -- & & -- & -- & \\
\hline
6 & & 10 & 14 & 8 & 6 & & 1 & & & 1 \\
& 5 & 9 & 1 & 1 & -- & 0 & & -- & -- & \\
\hline
7 & & 12 & 126 & 78 & 48 & & 1 & & & 0 \\
& 6 & 11 & 9 & 8 & -- & 1 & & 0 & 1 & \\
\hline
8 & & 14 & 1579 & 605 & 974 & & 36 & & & 1 \\
& 7 & 13 & 95 & 60 & -- & 35 & & 0 & 35 & \\
\hline
9 & & 16 & 26631 & 7557 & 19074 & & 883 & & & 1 \\ %"882"+1.
& 8 & 15 & 1515 & 602 & -- & 913 & & 31 & 882 & \\ %"882" verified.
\hline
\end{tabular}
\end{table}
%%%%%%%%%%%%%%%%%%%%%%%%%%%%%%%%%%%%%%%%%%%%%%%%%%

\begin{rem}\label{RemCountNgt2}
This reasoning covers all the connected graphs with specified number of vertices and edges, meaning that the valency $N(v)$ of every graph vertex $v$ can be any positive number (if $n > 1$). % RB: valency bounded from above by number of edges
By Lemma~\ref{LemmaHandshakesCancel} %from \S\ref{SecExamples} 
on p.~\pageref{LemmaHandshakesCancel}
it is seen that for the subspaces $V_{>2}$ of connected graphs restricted by $N(v) > 2$ for all $v$, the inclusion $\Id(V_{>2}) \subseteq V_{>2}$ holds. % RB: \geqslant or > / 2 or 3?
Therefore, the dimensions of cohomology groups for graphs \emph{with} such restriction on valency cannot exceed the dimension of respective cohomology groups for all the graphs under study (i.e.\ $N(v) > 0$).\footnote{\label{FootWhyValencyThree}%
Indeed, we recall that these cohomology dimensions --\,in the count with versus without restriction $N(v)>2$ of the valency\,-- are the same (e.g., see \cite[Proposition~3.4]{WillwacherGRT} with a sketch of the proof%(sic!)
).%
}
%RB: there are results about quasi-isomorphisms between graph complexes with more or fewer restrictions on the graphs (so that the cohomology is the same)
This means that trivial cohomology groups %from Table~\ref{TabCount} 
remain trivial under the extra assumption $N(v) > 2$ on valency; yet we already know the generators $\boldsymbol{\gamma}_3$, $\boldsymbol{\gamma}_5$, $\boldsymbol{\gamma}_7$,
and $[\boldsymbol{\gamma}_3,\boldsymbol{\gamma}_5]$ of all the nontrivial cohomology groups at~$n \leqslant 9$. This is confirmed in Table~\ref{TabCountNgt2}. 
%%%
\begin{table}[htb]
\caption{Dimensions of connected graph spaces with $N(v) > 2$ and dimensions of cohomology groups in bi\/-\/degree~$(n,2n-2)$.}\label{TabCountNgt2}
\begin{tabular}{|c | c | c | c | c | c | c | c | c | c| c |}
\hline
\multicolumn{2}{|c|}{$n$} & $\#E$ & $\#(\text{graphs})$ & $\#(=0)$ & \multicolumn{2}{|c|}{$\#(\neq 0)$, $N_\delta$} & \multicolumn{2}{|c|}{$N_{\text{ker}}$, $N_0$} & $N_{\text{im}}$ & $\dim H^*(n)$ \\
\hline
4 & & 6 & 1 & 0 & 1 & & 1 & & & 1 \\
& 3 & 5 & 0 & -- & -- & -- & & -- & -- & \\
\hline
5 & & 8 & 1 & 1 & 0 & & -- & & & 0 \\
& 4 & 7 & 0 & -- & -- & -- & & -- & -- & \\
\hline
6 & & 10 & 4 & 2 & 2 & & 1 & & & 1 \\
& 5 & 9 & 1 & 1 & -- & 0 & & -- & -- & \\
\hline
7 & & 12 & 18 & 12 & 6 & & 1 & & & 0 \\
& 6 & 11 & 5 & 4 & -- & 1 & & 0 & 1 & \\
\hline
8 & & 14 & 136 & 61 & 75 & & 11 & & & 1 \\
& 7 & 13 & 30 & 20 & -- & 10 & & 0 & 10 & \\
\hline
9 & & 16 & 1377 & 498 & 879 & & 164 & & & 1 \\
& 8 & 15 & 309 & 130 & -- & 179 & & 16 & 163 & \\
\hline
\end{tabular}
\end{table}

We finally note that the numbers of nonzero graphs with a specified number of vertices and edges (and $N(v) > 2$), which we list in Table~\ref{TabCountNgt2}, all coincide with the respective entries in Table~II in the paper \cite{WillwacherZivkovic2015Table}.
\end{rem}

%\subsection*{Discussion}
\begin{rem}\label{RemWaterfallGenerator}
We expect that there are many $\Id$-\/cocycles on $n$~vertices and $2n-2$~edges other than the ones containing the $(2\ell+1)$-\/wheel graphs (which Theorem~\ref{ThWillwacherWheels} provides) or their iterated commutators.
Namely, some terms in a weighted sum $\gamma\in\ker\Id$ can be disjoint graphs; moreover, the vertex\/-\/edge bi\/-\/grading of a connected component of a given term can be other than $(m,2m-2)$ for~$m\in\BBN$. Indeed, for any tuple of $\Id$-\/cocycles~$\gamma_i$ on $n_i$~vertices and $E_i$~edges satisfying $\sum_i n_i=n$ and $\sum_i E_i=2n-2$, one has that $\gamma\mathrel{{:}{=}}\bigsqcup_i \gamma_i\in\ker\Id$%; 
%e.g., recall that $N_0(8,15)=31$ (or~$16$ if $N(v)>2$), see Tables~\ref{TabCount} and~\ref{TabCountNgt2}
.
%%%
%(A count of the number of graphs at our disposal is contained in~\cite{WillwacherZivkovic2015Table}.)
%%%
%#(5,9)=0; #(6,11)=1; #(7,13)=10; #(8,15)=179, etc.
%%%
The graphs~$\gamma_i$ can be restricted by a requirement that each of them belongs to the domain of the orientation mapping~$\Or$, so that $\Or(\gamma)$~is a Kontsevich bi\/-\/vector graph (see~\cite{Ascona96} and~\cite{f16,WeFactorize5Wheel}). In this way new classes of generators of infinitesimal symmetries $\dot{\cP}=\Or(\gamma)(\cP)$ are obtained for %classical
Poisson structures~$\cP$.
\end{rem}

\subsection*{Acknowledgements}
The authors are grateful to M.~Kontsevich and T.~Willwacher for helpful discussion;
the authors thank the referees for criticism and advice.
This research %A.\,K. 
was supported in part by %by NWO grant VENI~639.031.623 (The Netherlands) and 
JBI~RUG project~106552 (Groningen, The Netherlands).
A part of this research was done while R.\,Buring and A.\,Kiselev were 
visiting at the~IH\'ES (Bures\/-\/sur\/-\/Yvette, France) and A.\,Kiselev 
was visiting at the~MPIM (Bonn) and Johannes Gutenberg\/--\/Uni\-ver\-si\-t\"at in
Mainz, Germany.

%??? explicit description of GRT Lie algebra generators ??? [Ref.]

\appendix
\section{The heptagon\/-\/wheel cocycle~$\boldsymbol{\gamma}_7$}\label{AppHepta}
\noindent%
In each term, the ordering of edges is lexicographic 
(cf.\ Table~\ref{TabHeptagonCocycle}).\label{pStartHepta}

%\begin{multline*}
\noindent%
$\boldsymbol{\gamma}_7={}$
\textbf{1}
\raisebox{-55pt}[87pt][55pt]{
\scalebox{.8}{
}
}%
.
\end{flushright}%
\label{pEndHepta}

\medskip
\noindent%
The sum of graphs~$\boldsymbol{\gamma}_7$ is a $\Id$-\/cocycle \emph{because}
when the differential~$\Id(\boldsymbol{\gamma}_7)$ is constructed, 
the images of many terms from~$\boldsymbol{\gamma}_7$ overlap in~$\Id(\boldsymbol{\gamma}_7)$
(by graphs on $9$~vertices and $15$~edges). Finding out what the resulting adjacency table is for the forty\/-\/six graphs in~$\boldsymbol{\gamma}_7$ and \,--more generally\,-- exploring whether such `meta\/-\/graphs', the vertices of which themselves are graphs that constitute $\Id$-\/cocycles %~$\gamma$ 
modulo coboundaries, are in any sense special, is an intriguing open problem.
(We claim that for~$\boldsymbol{\gamma}_7$, its meta\/-\/graph is connected.)

\newpage
\section{\textsc{Sage} code for the graph differential}\label{AppSage}
\noindent%
The following script, written in \textsc{Sage} version~7.2, can 
calculate the differential of an arbitrary sum of non\/-\/oriented graphs with a specified ordering on the set of edges for every term, and reduce sums of graphs modulo vertex and edge labelling.\footnote{Another software package for numeric computation of the graph complex cohomology groups in various degrees and loop orders is available from %\\
%\centerline
{\texttt{https://github.com/wilthoma/GHoL}.}%
}
As an illustration, it is shown how this can be used to find cocycles in the graph complex.
%; in Appendix~\ref{AppSage} we provide a program (in \textsc{Sage}) that calculates the differential of a given non-oriented graph  .
%This program code 

\tiny
\begin{verbatim}
import itertools

def insert(user, victim, position):
    result = []
    victim = victim.relabel({k : k + position - 1 for k in victim.vertices()},
                            inplace=False)
    victim = victim.copy(immutable=False)
    for edge in victim.edges():
        victim.set_edge_label(edge[0], edge[1], edge[2] + len(user.edges()))
    user = user.relabel({k : k if k <= position else k + len(victim) - 1 for k in user.vertices()},
                        inplace=False)
    for attachment in itertools.product(victim, repeat=len(user.edges_incident(position))):
        new_graph = user.union(victim)
        edges_in = user.edges_incident(position)
        new_graph.delete_edges(edges_in)
        new_edges = [(k if a == position else a, k if b == position else b, c)
                     for ((a,b,c), k) in zip(edges_in, attachment)]
        new_graph.add_edges(new_edges)
        result.append((1, new_graph))
    return result

def graph_bracket(graph1, graph2):
    result = []
    for v in graph2:
        result.extend(insert(graph2, graph1, v))
    sign_factor = 1 if len(graph1.edges()) % 2 == 1 and len(graph2.edges()) % 2 == 1 else -1
    for v in graph1:
        result.extend([(sign_factor*c, g) for (c,g) in insert(graph1, graph2, v)])
    return result
    
def graph_differential(graph):
    edge = Graph([(1,2,1)])
    return graph_bracket(edge, graph)

def differential(graph_sum):
    result = []
    for (c,g) in graph_sum:
        result.extend([(c*d,h) for (d,h) in graph_differential(g)])
    return result

def is_zero(graph):
    for sigma in graph.automorphism_group():
        edge_permutation = Permutation([graph.edge_label(sigma(i), sigma(j))
                                        for (i,j,l) in sorted(graph.edges(), key=lambda (a,b,c): c)])
        if edge_permutation.sign() == -1:
            return True
    return False

def reduce(graph_sum):
    graph_table = {}
    for (c,g) in graph_sum:
        if is_zero(g): continue
        
        # canonically label vertices:
        g_canon, relabeling = g.canonical_label(certify=True)
        # shift labeling up by one:
        g_canon.relabel({k : k + 1 for k in g_canon.vertices()})
        
        # canonically label edges (keeping track of the edge permutation):
        count = 1
        edges_seen = set([])
        edge_relabeling = {}
        for v in g_canon:
            edges_in = sorted(g_canon.edges_incident(v), key = lambda (a,b,c): a if b == v else b)
            for e in edges_in:
                if frozenset([e[0], e[1]]) in edges_seen: continue
                edge_relabeling[count] = e[2]
                g_canon.set_edge_label(e[0], e[1], count)
                edges_seen.add(frozenset([e[0], e[1]]))
                count += 1

        permutation = Permutation([edge_relabeling[i] for i in range(1, len(g.edges())+1)])
        g_canon = g_canon.copy(immutable=True)
        if g_canon in graph_table:
            graph_table[g_canon] += permutation.sign()*c
        else:
            graph_table[g_canon] = permutation.sign()*c
    return [(graph_table[g], g) for g in graph_table if not graph_table[g] == 0]

# Examples of graphs:

def wheel(n):
    return Graph([(k, 1, k-1) for k in range(2, n+2)] + [(k, k+1 if k <= n else 2, n+k-1)
                  for k in range(2, n+2)])

tetrahedron = wheel(3)
fivewheel = wheel(5)

print "The differential of the tetrahedron is", reduce(graph_differential(tetrahedron))

# Finding all cocycles on 6 vertices and 10 edges:

n = 6
graph_list = list(filter(lambda G: G.is_connected() and len(G.edges()) == 2*n - 2, graphs(n)))
# shift labeling up by one
for g in graph_list:
    g.relabel({k : k+1 for k in g.vertices()})
    for (k, (i,j,_)) in enumerate(g.edges()):
        g.set_edge_label(i, j, k+1)
# build an ansatz for a cocycle, with undetermined coefficients
nonzeros = filter(lambda g: not is_zero(g), graph_list)
coeffs = [var('c%d' % k) for k in range(0, len(nonzeros))]
cocycle = zip(coeffs, nonzeros)
# calculate its differential and reduce it
d_cocycle = []
for cocycle_term in cocycle:
    d_cocycle.extend(reduce(differential([cocycle_term])))
d_cocycle = reduce(d_cocycle)
# set the coefficients of the graphs in the reduced sum to zero, and solve
linsys = []
for (c,g) in d_cocycle:
    linsys.append(c==0)
print solve(linsys, coeffs)
\end{verbatim}

%\enlargethispage{1.2\baselineskip}
%{\small %
\normalsize%\smallskip
\noindent%
We finally recall that, to the best of our knowledge, the routines by McKay~\cite{BarNatan} for graph automorphism computation are now used in \textsc{SAGE} (hence by the above program).

%}

\begin{thebibliography}{10}

\bibitem{BarNatan}
\by{Bar\/-\/Natan D., McKay B.\,D.} (2001) Graph cohomology --- An overview and some computations, 13~p. (\jour{unpublished}),
\texttt{http://www.math.toronto.edu/\symbol{"7E}drorbn/papers/GCOC/GCOC.ps}

\bibitem{f16}
\by{Bouisaghouane A., Buring R., Kiselev A.} (2017)
The Kontsevich tetrahedral flow revisited, \jour{J.~Geom.\ Phys.} 
\vol{119}%September 2017 issue.
, 272--285.\ %20${}+{}$ix~p.
(\jour{Preprint} \texttt{arXiv:1608.01710} %(v4)
[q-alg])

\bibitem{tetra16}
\by{Bouisaghouane A., Kiselev A.\,V.} (2017)
Do the Kon\-tse\-vich tetrahedral flows preserve or destroy the space of Poisson bi\/-\/vectors\,?
%\jour{Pre\-print} $\smash{\text{IH\'ES}}$/M/16/12 (Bures\/-\/sur\/-\/Yvette, 
%France), %12~p.
\jour{J.~Phys.\textup{:}\ Conf.\ Ser.} \vol{804} 
Proc.\ XXIV Int.\ conf.\
`Integrable Systems and Quantum Symmetries' (14--18 June 2016, $\smash{\text{\v{C}VUT}}$ Prague,
Czech Republic), Paper~012008, 10~p.\ %\texttt{arXiv:1609.06677} [q-alg]
(\jour{Pre\-print} 
%$\smash{\text{IH\'ES}}$/M/16/12 (Bures\/-\/sur\/-\/Yvette, %2015
%France), 
\texttt{arXiv:1609.06677} [q-alg])

\bibitem{Brown}
\by{Brown F.} (2012) Mixed Tate motives over~$\mathbb{Z}$, 
\jour{Ann.\ Math.~(2)} \vol{175}:2, 949--976.
%11S20 (11M32 14F42)

\bibitem{sqs15}
\by{Buring R., Kiselev A.\,V.} (2017) On the Kon\-tse\-vich $\star$-\/product asso\-ci\-a\-ti\-vi\-ty me\-cha\-n\-ism, \jour{PEPAN Letters} \vol{14}:2, 403--407.\ %
(\jour{Preprint} \texttt{arXiv:1602.09036} [q-alg])

\bibitem{cpp}
\by{Buring R., Kiselev A.\,V.} (2017) %Software modules 
The expansion $\star$ mod~$\bar{o}(\hbar^4)$
and computer\/-\/assisted proof schemes in the Kon\-tse\-vich deformation quantization,
\jour{Preprint} $\smash{\text{IH\'ES}}$/M/17/05, %50${}+{}$xvi~p.
\texttt{arXiv:1702.00681}~%(v2)
[math.CO], 67~p.
%%%%%%%%%%%%%%%%%%%%%%%%%%%%%%%%%%%%%%%%%%%%5
%\jour{Preprint} \texttt{arXiv:1702.00681} [math.CO], 44~+ xvi~p.
%see link:\\
%\texttt{https://github.com/rburing/kontsevich\symbol{"5F}graph\symbol{"5F}series-cpp}

\bibitem{WeFactorize5Wheel}
\by{Buring R., Kiselev A.\,V., Rutten N.\,J.} (2017)
The Kontsevich\/--\/Willwacher pentagon\/-\/wheel symmetry of %classical
Poisson structures%: construction and approbation/verification%
, SDSP~IV (12--16~June 2017, $\smash{\text{\v{C}VUT}}$ D\v{e}\v{c}\'\i{}n, Czech Republic)%,
%in preparation
.

\bibitem{DolgushevRogersWillwacher}
\by{Dolgushev V. A., Rogers C. L., Willwacher T. H.} (2015) Kontsevich's graph complex, GRT, and the deformation complex of the sheaf of polyvector fields,
\jour{Ann.\ Math.} \vol{182}:3, 855--943.\ %
(\jour{Preprint} \texttt{arXiv:1211.4230} [math.KT])

\bibitem{DrinfeldGRT1990}
\by{Drinfel'd V. G.} (1990)
On quasitriangular quasi\/-\/Hopf algebras and on a group that is closely connected
with $\text{Gal}(\overline{\mathbb{Q}}/\mathbb{Q})$, 
\jour{Algebra i Analiz} \vol{2}:4, 149--181 (in Russian); Eng.\ transl.\ in:
\jour{Leningrad Math.~J.} \vol{2}:4, 829--860 (1990).
%16W30 (17B37)

\bibitem{MKParisECM}%{Kontsevich1994}
\by{Kontsevich M.} (1994)
Feynman diagrams and low\/-\/dimensional topology,
\book{First Europ.\ Congr.\ of Math.} \vol{2} (Paris, 1992), 
Progr.\ Math. \vol{120}, Birkh\"auser, Basel, 97--121. 
    %(Reviewer: Anatoly Libgober) 57R57 (14H15 32G15 57M25)

\bibitem{MKZurichICM}
\by{Kontsevich M.} (1995)
Homological algebra of mirror symmetry, 
\book{Proc.\ Intern. Congr. Math.}~\vol{1} %,\,\vol{2} 
(Z\"urich, 1994), 
Birkh\"auser, Basel, 120--139.

\bibitem{Ascona96}
\by{Kontsevich M.} (1997)
Formality conjecture. 
\book{Deformation theory and symplectic geometry} (Ascona %June 17--21,
1996, D.\,%aniel 
Sternheimer, J.\,%ohn 
Rawnsley and S.\,%imone 
Gutt, eds), 
Math.\ Phys.\ Stud.~\vol{20}, Kluwer Acad.\ Publ., Dordrecht, 139--156.

\bibitem{Kontsevich2017Bourbaki}
\by{Kontsevich M.} (2017) Derived Grothendieck\/--\/Teichm\"uller group and graph complexes [after T.~Will\-wa\-cher], \jour{S\'eminaire Bourbaki} (69\`eme ann\'ee, Janvier 2017), no.~1126, 26~p.

\bibitem{KhoroshkinWillwacherZivkovic}
\by{Khoroshkin A., Willwacher T., \v{Z}ivkovi\'c M.} (2017) Differentials on graph complexes, \jour{Adv.\ Math.} \vol{307}, 1184--1214.\ %
(\jour{Preprint} \texttt{arXiv:1411.2369} [q-alg])

\bibitem{WillwacherRossi14042047}
\by{Rossi C.\,A., Willwacher T.} (2014)
P.~Etingof's conjecture about Drinfeld associators,
\jour{Preprint} \texttt{arXiv:1404.2047} [q-alg], 47~p.

\bibitem{WillwacherGRT}
\by{Willwacher T.} (2015) M.~Kontsevich's graph complex and the Grothendieck\/--\/Teichm\"uller Lie algebra,
\jour{Invent.\ Math.} \vol{200}:3, 671--760.\ %
(\jour{Preprint} \texttt{arXiv:1009.1654} [q-alg]) %v4 (2013)

\bibitem{WillwacherZivkovic2015Table}
\by{Willwacher T., \v{Z}ivkovi\'{c} M.} (2015)
Multiple edges in M.~Kontsevich's graph complexes and computations of the dimensions and Euler characteristics, \jour{Adv.\ Math.} \vol{272}, 553--578.\ %
(\jour{Preprint} \texttt{arXiv:1401.4974} %v2 
[q-alg])
%05A15 (18G35)

\end{thebibliography}
\end{document}